\documentclass[hidelinks,onefignum,onetabnum]{siamart220329}
\usepackage{xr-hyper}
\usepackage{hyperref}

\usepackage{lipsum}
\usepackage{amsfonts,amsmath,amssymb}
\usepackage{enumerate}
\usepackage{graphicx}
\usepackage{subcaption}
\usepackage{epstopdf}
\usepackage{algorithmic}
\usepackage{bbm}
\usepackage{xr-hyper}
\usepackage{hyperref}
\ifpdf
  \DeclareGraphicsExtensions{.eps,.pdf,.png,.jpg}
\else
  \DeclareGraphicsExtensions{.eps}
\fi

\newsiamremark{remark}{Remark}
\newsiamremark{hypothesis}{Hypothesis}
\crefname{hypothesis}{Hypothesis}{Hypotheses}
\newsiamthm{claim}{Claim}

\headers{From habitat decline to collapse}{Y. Salmaniw, Z. Shen and H. Wang}

\title{From habitat decline to collapse: a spatially explicit approach connecting habitat degradation to destruction\thanks{
\funding{YS was partially supported by an NSERC Doctoral Fellowship and by Postdoctoral Fellowship (NSERC Grant PDF-578181-2023). ZS was partially supported by a start-up grant from the University of Alberta, NSERC RGPIN-2018-04371 and NSERC DGECR-2018-00353. HW was partially supported by an NSERC Individual Discovery Grant RGPIN-2020-03911 and an NSERC Discovery Accelerator Supplement
Award RGPAS-2020-00090 as well as a Tier 1 Canada Research Chair Award.}}}

\author{Yurij Salmaniw\thanks{Mathematical Institute, University of Oxford, Woodstock Road, Oxford, OX2 6GG, United Kingdom (\email{yurij.salmaniw@maths.ox.ac.uk})}
\and
Zhongwei Shen\thanks{Department of Mathematical and Statistical Sciences, University of Alberta, Edmonton, Canada, T6G 2G1 (\email{zhongwei@ualberta.ca})}
\and
Hao Wang\thanks{Department of Mathematical and Statistical Sciences, University of Alberta, Edmonton, Canada, T6G 2G1(\email{hao8@ualberta.ca})}
}

\newtheorem{thm}{Theorem}[section]
\newtheorem{lem}[thm]{Lemma}
\newtheorem{cor}[thm]{Corollary}
\newtheorem{prop}[thm]{Proposition}
\newtheorem{defn}{Definition}

\newtheorem{assum}{Assumption}

\renewtheorem{remark}[thm]{Remark}

\def \a    {\alpha}

\def \D    {\Delta}
\def \e    {\varepsilon}

\def \i    {\infty}
\def \l    {\lambda}

\def \oo {\"{o}}
\def \O    {\Omega}

\def \p    {\partial}

\def \s    {\sigma}

\def \Ind {\mathbbm{1}}

\DeclareMathOperator*{\supp}{supp}

\newcommand{\as}[1]{\left\vert#1\right\vert}

\newcommand{\magg}[1]{\left\vert#1\right\vert ^{2}}
\newcommand{\norm}[1]{\left\Vert#1\right\Vert}
\newcommand{\normm}[1]{{\left\vert\kern-0.25ex\left\vert\kern-0.25ex\left\vert #1 
    \right\vert\kern-0.25ex\right\vert\kern-0.25ex\right\vert}}
\newcommand{\beginc}{\begin{center}}
\newcommand{\cend}{\end{center}}
\newcommand{\eq}[1]{\begin{align*}#1\end{align*}}
\newcommand{\eql}[2]{\begin{align}\label{#2}#1\end{align}}
\newcommand{\grad}{\nabla}

\newcommand{\nn}{\nonumber}

\ifpdf
\hypersetup{
  pdftitle={From habitat decline to collapse: a spatially explicit approach connecting habitat degradation to destruction},
  pdfauthor={Y. Salmaniw, Z. Shen and H. Wang}
}
\fi

\begin{document}

\maketitle

% REQUIRED
\begin{abstract}
Habitat loss, driven primarily by anthropogenic activity, significantly threatens ecosystem sustainability. While it is well understood that habitat loss is the leading contributor to declines in biodiversity worldwide, the connection between habitat degradation, destruction, and different locomotion strategies remains unclear. We use a reaction-diffusion framework to analyze the effects of habitat loss on population persistence and abundance. We establish necessary and sufficient conditions for the existence of an extinction threshold, beyond which further degradation of the environment predicts deterministic extirpation. Our results offer a robust analytical connection between habitat degradation and destruction, providing a mechanistic understanding of species persistence under varying environmental conditions and differing locomotion strategies.
\end{abstract}

% REQUIRED

% REQUIRED

\section{Introduction}

Habitat loss, primarily driven by human activities such as agriculture, urbanization, and resource extraction, poses significant threats to biodiversity and ecosystem sustainability \cite{Heinrichs2016,Fischer2007,Pimm2014,Diaz2019}. These activities result in the degradation and destruction of natural habitats, leading to species extinction or displacement, facilitating the invasion of more aggressive species, and reducing recolonization abilities \cite{Diaz2019,Chase2020}. The societal impacts of habitat loss include declines in ecosystem services, climate regulation, and political stability \cite{myers1993ultimate}, as well as economic disruptions and increased refugee movements \cite{Balmford2002,myers1993environmental}.

While nature's intrinsic value is undeniable \cite{Ehrlich1981}, our dependence on ecosystem services \cite{Daily1997} underscores the importance of investigating various forms of habitat loss, including degradation, destruction, and fragmentation, to mitigate negative long term consequences. A primary goal of the present work is to establish a robust connection between habitat degradation and destruction in a mechanistic framework, and to understand the role species-specific traits play in population persistence. To compliment our primary goal, we are also interested in understanding the timescales at which the impacts of habitat loss are realized.

\subsection{Motivation \& key concepts}\label{sec:intro}
\hspace{1cm}\\
\noindent\textbf{Disentangling landscape changes:} As a general concept, habitat loss includes combined processes of habitat degradation, destruction, and fragmentation. Realistically, these components are closely intertwined, making it difficult to separate the relative impacts each hold \cite{Laurance1}. This has led to much debate on whether concepts such as fragmentation are ambiguous or even meaningless \cite{Laurance1, Fahrig2017}, whether fragmentation is beneficial or detrimental to biodiversity \cite{fletcher2018habitat, fahrig2019habitat}, and more recently over which habitat configurations (e.g., ``\textit{single large or several small" or ``\textit{single large AND several small}"}) might be optimal to maintain biodiversity \cite{Fahrig2017, SZANGOLIES2022}. Mechanistic modelling offers an approach that allows one to more easily isolate confounding factors, providing significantly more insight into when we should expect particular landscape changes to be more or less detrimental to the local population(s) considered. Here, we focus on two key questions: 
\begin{enumerate}[i.)]
    \item What is the connection between habitat \textit{degradation} and habitat \textit{destruction}?
    \item What role do species-specific traits play in population persistence, abundance, and biodiversity? 
\end{enumerate}
Connected to each of these is a question of the timescales over which impacts are realized. We seek to answer these questions in an ecologically meaningful way, while retaining the ability to isolate different aspects of habitat loss. To this end, we first review some key concepts to be used in the development of our model.

\noindent\textbf{The habitat question:} Habitat degradation refers to processes that decrease habitat quality, such as moderate pollution or selective logging \cite{Heinrichs2016}. In contrast, habitat destruction occurs when habitat alterations prevent species from sustaining themselves, such as heavy pollution or through clear-cut logging \cite{Laurance1}. Precise definitions of these differing but intimately related concepts are crucial for modeling these ecological processes accurately, and to appropriately interpret the results obtained \cite{Salmaniw2022,hall1997habitat,Fahrig2003,Fischer2007}. Importantly, we observe that the concept of habitat is implicit in our understanding of its loss.

In this work we define \textit{habitat} as the resources and conditions that support the occupancy, survival, and reproduction of a given organism. Habitat is organism-specific; it relates the presence of a species, population, or individual (animal or plant) to an area’s physical and biological characteristics \cite{hall1997habitat}. \textit{Habitat degradation} refers generally to any process that diminishes habitat quality \cite{Heinrichs2016, Laurance1}, while \textit{habitat destruction} occurs when such alterations render a habitat incapable of supporting its original species \cite{Laurance1}. Rather than viewing a habitat as degraded OR destroyed, we opt for understanding these concepts as belonging to a continuous spectrum with an intact habitat at one end, and a destroyed habitat at the other.
$$
\Large
\xrightarrow[\text{\large State of the habitat under increased degradation}]{\text{\normalsize intact} \hspace{3cm} \text{\normalsize degraded} \hspace{3cm} \text{\normalsize destroyed}}
$$

\noindent\textbf{Species-specific traits:} To fully understand habitat loss requires acknowledging that habitat is specific to the traits of each organism, making mechanistic models that incorporate species-specific traits essential. Therefore, habitat itself encompasses more than just vegetation type or structure; it represents ``the sum of the specific resources that are needed by organisms [to survive and reproduce]"\cite{hall1997habitat}. Identical landscapes can serve as habitats for some species but not for others. Meaningful investigation of habitat loss should consider species-specific mechanistic models alongside empirical data, which can provide more robust insights given the high cost and extrapolation limitations of data collection.

Motivated by such considerations, this work has two primary objectives: first, to establish a robust analytical connection between habitat degradation and destruction within a spatially explicit framework; second, to develop a model that addresses timescales and rates of convergence in time-dependent problems involving habitat alterations. Explicit in our model is a consideration of species-specific traits through local intrinsic growth rates and different locomotion strategies. Understanding habitat loss as a transition from degradation to destruction requires examining both spatial and temporal scales, and the relative impacts of different forms of habitat loss \cite{Heinrichs2016, Chase2020}. The notion of an extinction debt \cite{Tilman1994} further motivates this study by highlighting the delayed effects of habitat loss, which can lead to extinction generations after initial habitat alteration \cite{Pimm2000}. Thus, it is crucial to comprehend both the immediate and long-term effects of habitat loss and the timescales over which these processes operate \cite{Pimm2000}.

\subsection{Model formulation}\label{sec:intro_2}

We extend the framework of \cite{Salmaniw2022} by considering habitat destruction as a limiting case of habitat degradation.  This perspective will allow a precise and robust analytical connection between habitat degradation and destruction. We first consider a spatially heterogeneous intrinsic growth rate to describe population growth when a subregion of the habitat experiences some level of degradation. The landscape $\O \subset \mathbb{R}^{N}$ ($N\geq1$) is thus partitioned into two subregions, $B$ and $\Omega\setminus B$, where $B$ denotes the degraded region and $\Omega\setminus B$ the undisturbed region. Population growth in $\O\setminus B$ follows a logistic-type functional response $f(x, u)$, while in B the population declines at a constant rate $c \geq 0$. The functional response over $\Omega$ is succinctly written as $\Ind_{\O \setminus B}(x) f(x,u) - c \Ind_B (x) u$,  where $\Ind_K(x)$ is the indicator function of a set $K\subset\mathbb{R}^{N}$.

In \cite{Salmaniw2022}, the form $f(x,u) = u(1-u)$ was used as a prototypical growth term for a habitat degradation model with a zero-flux (homogeneous Neumann) boundary condition along $\p \O$. We generalize this degradation model as follows:
\begin{equation}\label{scalareqn-2}
\begin{cases}
u_t = d \D u + \Ind_{\O\setminus B} f(x,u) - c \Ind_{B} u, & \text{in}\quad \O\times (0,\i), \\
\frac{\p u}{\p \nu } = 0, &\text{on}\quad \p \O \times (0,\i),
\end{cases}
\end{equation}
where $\p / \p \nu$ denotes the outward facing unit normal vector, and $B \subset \O$ and $f$ are assumed to respectively satisfy Assumptions \ref{assumptionA}-\ref{assumptionf} (see Subsection \ref{subsec-assumption}). Different from \cite{Salmaniw2022}, this allows for the possibility of heterogeneity in the undisturbed region $\O \setminus B$. 

We then formulate a habitat \textit{destruction} problem as
\begin{equation}\label{scalareqn-1}
    \begin{cases}
u_t = d \D u + f(x,u), & \text{in}\quad \O \setminus \overline{B} \times (0,\infty), \\
\frac{\p u}{\p \nu } = 0, & \text{on}\quad  \p \O \times (0,\infty), \\
u = 0, & \text{on}\quad \p B \times (0,\infty).
\end{cases}
\end{equation}
Here, the habitat destruction problem is described by a reaction-diffusion equation with a homogeneous Neumann boundary condition on the outer boundary $\partial \Omega$ and a homogeneous Dirichlet boundary condition along $\partial B$, representing hostile regions within the undisturbed region $\Omega$. The solution to problem \eqref{scalareqn-1} serves as the limit candidate as $c \rightarrow +\infty$ in the habitat degradation problem \eqref{scalareqn-2}.

%%%%%%%%%%%%%%%%%%%%%%%%%%%%%

\subsection{Main results}

It is well-known (see, e.g., \cite{Hess1991, Zhao2017}) that the principal spectral theory of eigenvalue problems associated with the linearization of \eqref{scalareqn-2} and \eqref{scalareqn-1} about their respective trivial states play a crucial role in understanding their long-term dynamics. We introduce them here and refer the reader to the Supplementary Materials \ref{appendix-1} for a more general consideration of related concepts and results.

The eigenvalue problem associated with the linearization of \eqref{scalareqn-2} about the trivial state reads
\begin{equation}\label{MainEig2-2}
\begin{cases}
d \D \phi +  m_{c} \phi + \mu \phi = 0 ,& \text{in}\quad \O ,  \\
\frac{\p \phi}{\p \nu } = 0 , &\text{on}\quad \p \O,
\end{cases}
\end{equation}
where $m_{c} := \Ind_{\O \setminus B}f_u (\cdot,0) - c \Ind_{B}$. Since $m_{c}\in L^{\infty}(\Omega)$ (see Assumption \ref{assumptionf}), Proposition \ref{deglineig-app} applies to \eqref{MainEig2-2} for each $c$ fixed. We say that a problem has a \emph{principal eigenvalue} if it has a positive eigenfunction. Denote by $\mu_{1,c}$ the principal eigenvalue of \eqref{MainEig2-2}, and by $\phi_{1,c}$ its eigenfunction. 

The eigenvalue problem associated with the linearization of \eqref{scalareqn-1} about the trivial state reads
\begin{equation}\label{MainEig1-2}
\begin{cases}
d \D \phi +  f_u (\cdot,0)\phi + \mu \phi = 0 ,& \text{in}\quad \O \setminus \overline{B},  \\
\frac{\p \phi}{\p \nu } = 0 ,& \text{on}\quad \p \O , \\
\phi = 0 ,&  \text{on}\quad \p B. 
\end{cases}
\end{equation}
As $f_u (\cdot,0)\in L^{\infty}(\Omega\setminus\overline{B})$ (see Assumption \ref{assumptionf}), Proposition \ref{deglineig-1-app} applies to \eqref{MainEig1-2}. Denote by $\mu_{1,\infty}$ the principal eigenvalue of \eqref{MainEig1-2}, and by $\phi_{1,\infty}$ its eigenfunction.

Our first result connects the principal eigenpairs of \eqref{MainEig2-2} and \eqref{MainEig1-2} as $c\to\infty$.

\begin{thm}\label{convergence-thm-2}
The following hold.
\begin{enumerate}[\rm(1)]
\item The function $c\mapsto\mu_{1,c}$ is strictly increasing on $(0,\infty)$, and $\lim_{c \to \i}\mu_{1,c} =\mu_{1, \infty}$.

\item $\lim_{c \to \i}\phi_{1,c} = \phi_{1,\infty}$ in $H^1 (\O)$ under the normalization \\ $\|\phi_{1,c}\|_{L^{2}(\Omega)} = \|\phi_{1,\infty}\|_{L^{2}(\O\setminus \overline{B})} = 1$. 
\end{enumerate}
\end{thm}

In a biological setting $-\mu_{1,c}$ can be used to describe the average growth rate of the population for small population sizes \cite{CantrellCosner2003}, and so Theorem \ref{convergence-thm-2} suggests the intuitive insight that degrading the habitat affects the population growth rate in a monotonic way. In fact, we are able to prove the following.

\begin{thm}\label{convlem4}
Assume $\mu_{1,\i} < 0$. Then, \eqref{scalareqn-1} admits a unique positive steady state $u^*_{\infty}$ and \eqref{scalareqn-2} admits a unique positive steady state $u_c^*$ for all $c\gg1$. Moreover, 
$$
\lim_{c \to \i}u_c ^* = u^*_{\infty}\quad\text{in}\quad C( \overline{\O} ).
$$
\end{thm}

Note that $-\mu_{1,\infty}>0$ implies that the average population growth rate is positive for any level of degradation. We point out that in Theorem \ref{convlem4}, the case $\mu_{1,\infty}>0$ is of little interest as $0$ is the only steady state to \eqref{scalareqn-1} and therefore also to \eqref{scalareqn-2} for all $c\gg1$. 

\begin{thm}\label{finalscalarconvergence}
Assume that $\mu_{1, \i} \neq 0$. Let $u_c$ and $u_{\infty}$ be the unique solutions to problems \eqref{scalareqn-2} and \eqref{scalareqn-1}, respectively, with the initial data satisfying $0 \lneq u_{c}(\cdot,0)=u_{\infty}(\cdot,0) \in C_B ^1 (\overline{\O})$ and $\supp (u_\infty (\cdot,0)) \Subset \overline{\O} \setminus \overline{B}$. Then, 
$$
\lim_{c\to\infty}u_c = u_{\infty}\quad\text{uniformly in}\quad\overline{\O} \times [0,\i). 
$$
\end{thm}

\begin{remark} Convergence in each of these results holds in the sense that solutions to the degradation problem converge uniformly in $\overline{\O} \setminus B$ to the solution of the associated destruction problem while converging uniformly to zero in the set $\overline{B}$. Therefore, we identify the solutions to the destruction problems with their (continuous) extension by zero in the set $B$. This convention is assumed throughout the remainder of this paper.
\end{remark}

In proving Theorem \ref{finalscalarconvergence}, one of the major difficulties is the uniform convergence on a large time interval $(T,\i)$, independent of $c$. This challenge is overcome proving first a uniform convergence result in an arbitrary time interval $(0,T]$, followed by a careful use of asymptotic stability results for the degradation and destruction problems (see Theorems \ref{globaldegradation} and \ref{globaldestruction}), which are shown to hold in a uniform sense with respect to parameter $c$ when $\mu_{1,\infty} \neq 0$. 

We note that the restriction on the initial data is a technical one. To achieve uniform results across the parameter $c$, we need control over the comparison between the initial data and the corresponding eigenfunction. While such restrictions are not required for any $c\geq0$ fixed, they become challenging to manage at the boundary of set $B$ for arbitrary $c$. 

\subsection{Applications}\label{sec:apps}

\noindent\textbf{General trends:} Theorems \ref{convlem4}-\ref{finalscalarconvergence} establish a direct connection between habitat degradation and destruction, providing a complete answer to question i.) introduced earlier. These theorems show that arbitrary levels of degradation do not necessarily displace local species—even in the limit as $c \to +\infty$, the population may persist. However, as the area of the degraded/destroyed region increases, the chances of survival decrease due to the monotonicity of the principal eigenvalue with respect to subregions of destroyed habitat (see Proposition \ref{existthm-app}(i) and Proposition \ref{deglineig-1-app}(iv)). Moreover, this connection is ordered in an intuitive way: given solutions $u_c$, $c>0$, with the same initial data chosen from $C^{+}(\overline{\Omega})\setminus\{0\}$, $c_2 > c_1$ implies that $u_{c_2}< u_{c_1}$ in $\overline{\O} \times (0,\i)$. (see Lemma \ref{monotonicallydecreasing}). This monotonic behavior is insightful in conjunction with empirical studies on habitat loss, whose results can sometimes be difficult to interpret, e.g., when it is sometimes observed that habitat modification is \textit{beneficial} to certain populations \cite{fletcher2018habitat, fahrig2019habitat}. Our results suggest that any positive effect a local species may experience in a region of altered habitat must come from another process, such as that of fragmentation or species-species interactions, and may be influenced more directly by more complicated factors not explicitly included here, such as an \textit{edge effect} \cite{Fahrig2013, Ewers2007}. 

\noindent\textbf{Tipping points:} A tipping point, typically used ``\textit{loosely as a metaphor for the phenomenon that, beyond a certain threshold, runaway change propels a system to a new state}"\cite{VANNES20169}, is an increasingly important concept in ecology. Of particular interest is the identification of so-called \textit{early warning signals} that precede such tipping points, though it has been acknowledged that truly generic warning signals are unlikely to exist \cite{Boettiger2013}. One possible remedy, particularly in the absence of good data and controlled replicates, is to study transitions specific to real systems, and to ``\textit{model the expected behaviour of a stable ecosystem}" \cite{Boettiger2012, Boettiger2013}. As a Corollary of our main results, we deduce necessary and sufficient conditions for the existence of an extinction threshold for problem \eqref{scalareqn-2} with respect to the parameter $c$ in the following sense. Let $u_c(x,t)$ denote the unique solution to problem \eqref{scalareqn-2} with initial data $u_c(\cdot, 0) \in C^+(\Omega) \setminus {0}$.
\begin{defn}\label{def-extinction-threshold}
    We call $c_0 \in (0, \infty)$ an \textup{extinction threshold} for problem \eqref{scalareqn-2} if there holds
    $$
\liminf_{t \to \infty} u_c(\cdot,t)\gneq  0, \quad\forall\,  0< c < c_0,
    $$
    while
    $$
\lim_{t \to \infty} u_c(\cdot,t) = 0, \quad \forall\, c > c_0.
    $$
\end{defn}
\noindent Put simply, an extinction threshold provides a single value beyond which deterministic extinction of the population is predicted. Therefore, assuming population persistence is the desired outcome, increasing the size of the extinction threshold is of particular interest. For model \eqref{scalareqn-2} we have the following result.
\begin{cor}\label{cor:exthresh}
Problem \eqref{scalareqn-2} admits an extinction threshold $c_{0}\in(0,\infty)$ if and only if $\mu_{1,\infty} > 0$. Moreover, in the case of $\mu_{1,\infty} > 0$, the following hold.
\begin{itemize}
    \item[\rm(1)] $c_0 > \frac{1}{\as{B}} \int_{\Omega \setminus B} f_u (x,0) {\rm d}x$.

    \item[\rm(2)] Denote by $u_{c}$ the unique solution to problem \eqref{scalareqn-2} with initial data $u_{c}(\cdot,0)\in C^{+}(\overline{\Omega})\setminus\{0\}$. Then,
    $$
\lim_{t\to\infty}u_{c}(\cdot,t)=\begin{cases}
    u_{c}^{*},&\text{if}\quad 0<c<c_0,\\
    0,&\text{if}\quad c>c_0
\end{cases}
\quad\text{in}\quad C(\overline{\Omega}).
$$

\item[\rm(3)] Suppose that $u_{c}(\cdot,0)$ is independent of $c$. Then, for any $\tilde c >c_0$, there exist $r = r(\tilde c)>0$ and $M>0$ (depending only on $\tilde{c}$ and the common initial data) such that
$$
 \sup_{c \geq \tilde c} \norm{u_c(\cdot,t)}_{C(\overline{\Omega})} \leq M e^{-r t},\quad\forall t\geq0.
$$
\end{itemize}
\end{cor}

Corollary \eqref{cor:exthresh} produces several key insights. First, by (1), we can always increase the extinction threshold $c_0$ by decreasing the size of the degraded region $B$, or by improving the habitat quality in the undisturbed region $\Omega \setminus B$. Less intuitively, perhaps, is the role of the species-specific intrinsic growth rate encoded in the term $f_u (x,0)$: the practicality of particular conservation efforts should align with expected success of the population in the remaining habitat regions. Second, by (2), we can identify an expected baseline of population abundance, allowing for a more robust analysis of the current state of a given population. For example, it is possible to measure a declining population abundance over time, but this alone does not provide insight into whether the population will \textit{continue} to decrease over time, or whether the population is expected to settle down at a new, perhaps lower, total abundance without observing extirpation. Finally, from (3), we are able to understand the timescales over which extirpation outcomes are realized. For cases considered here, we observe that the population is expected to decline exponentially; however, even an exponential rate of convergence can appear to be slow-moving if the rate $0< r \ll 1$ is particularly small. This is expected to be the case \textit{near} the extinction threshold $c_0$, where $r(\tilde c)$ is small for values near $c_0$.

\noindent\textbf{Species-specific traits and habitat fragmentation:} Recently, the modelling formulation used here and in \cite{Salmaniw2022} has been utilized as a mechanistic approach to understand the impacts of habitat fragmentation in a model ecosystem \cite{Zhang2024a, Zhang2024b}. We measured the total abundance of a population of \textit{C. Elegans} under differing habitat arrangements. We consider two strains, identical in their reproductive capabilities, but differing in their locomotion speeds. Regions of blank agar (region $B$ here) and blank agar with food (region $\Omega \setminus B$ here) were stamped onto a petridish, with a focus on the impact of fragmenting the environment. In \cite{Zhang2024a} a simplified one-dimensional model was used; in the forthcoming followup paper \cite{Zhang2024b}, we explore the impact of corridors, requiring a two-dimensional version of the model. Even with a simplistic approach, we observed that 1. habitat amount alone is insufficient to predict total abundance; 2. habitat fragmentation \textit{per se} \cite{Fahrig2017} has a significant impact on total abundance; 3. locomotion rates significantly affect the total abundance measured. In future efforts, we hope to explore 1. the impacts of destroyed habitat by including regions of copper (known to be toxic to \textit{C. Elegans}), and 2. to explore the impacts of competition between two strains with different locomotion strategies.

\begin{figure}
\centering
\begin{subfigure}{.475\textwidth}
  \centering
  \includegraphics[width=1.0\linewidth]{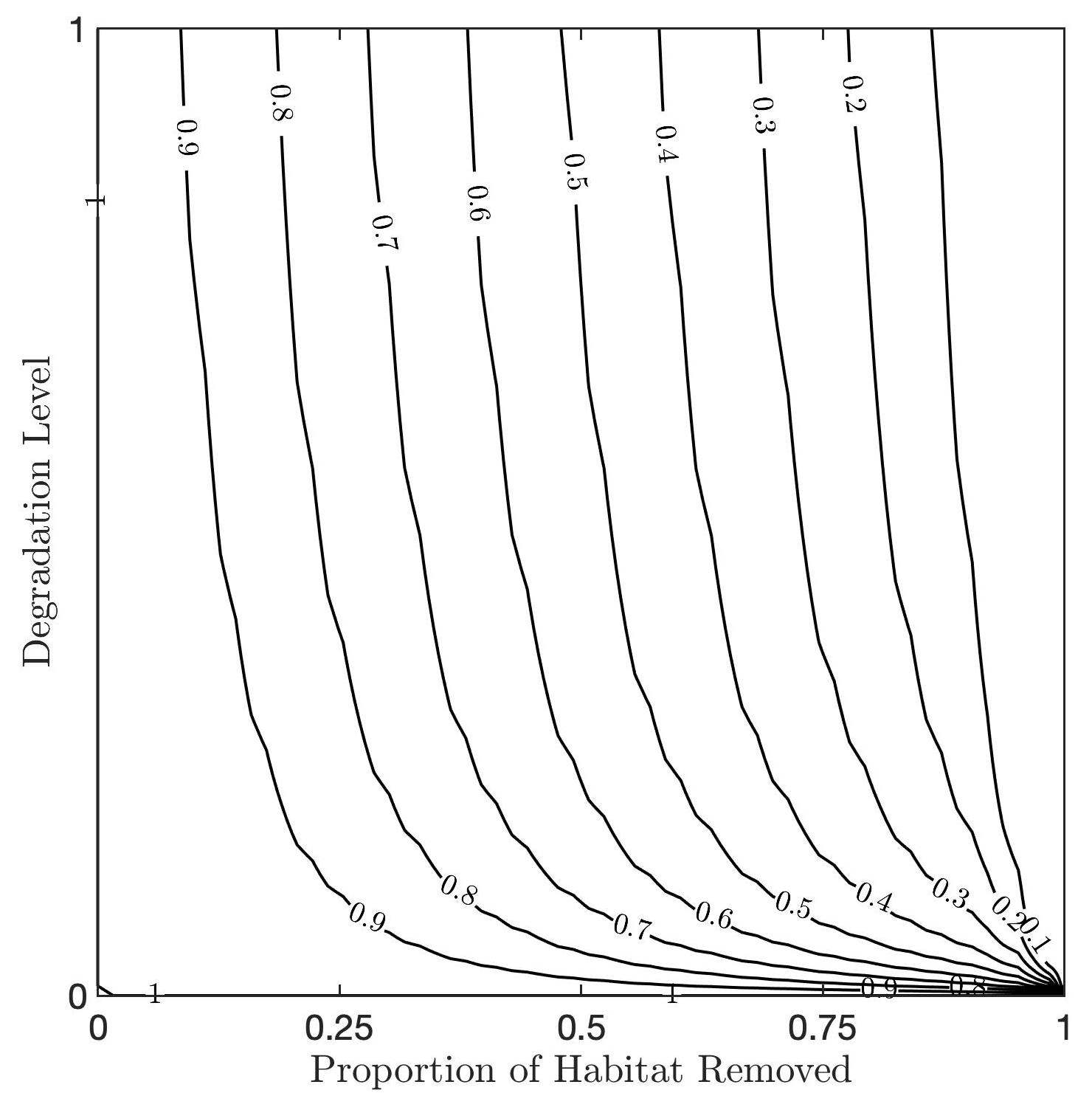}
  \caption{Diffusion fixed at $d = 1.0$.}
\end{subfigure}%
\begin{subfigure}{.475\textwidth}
  \centering
  \includegraphics[width=1.0\linewidth]{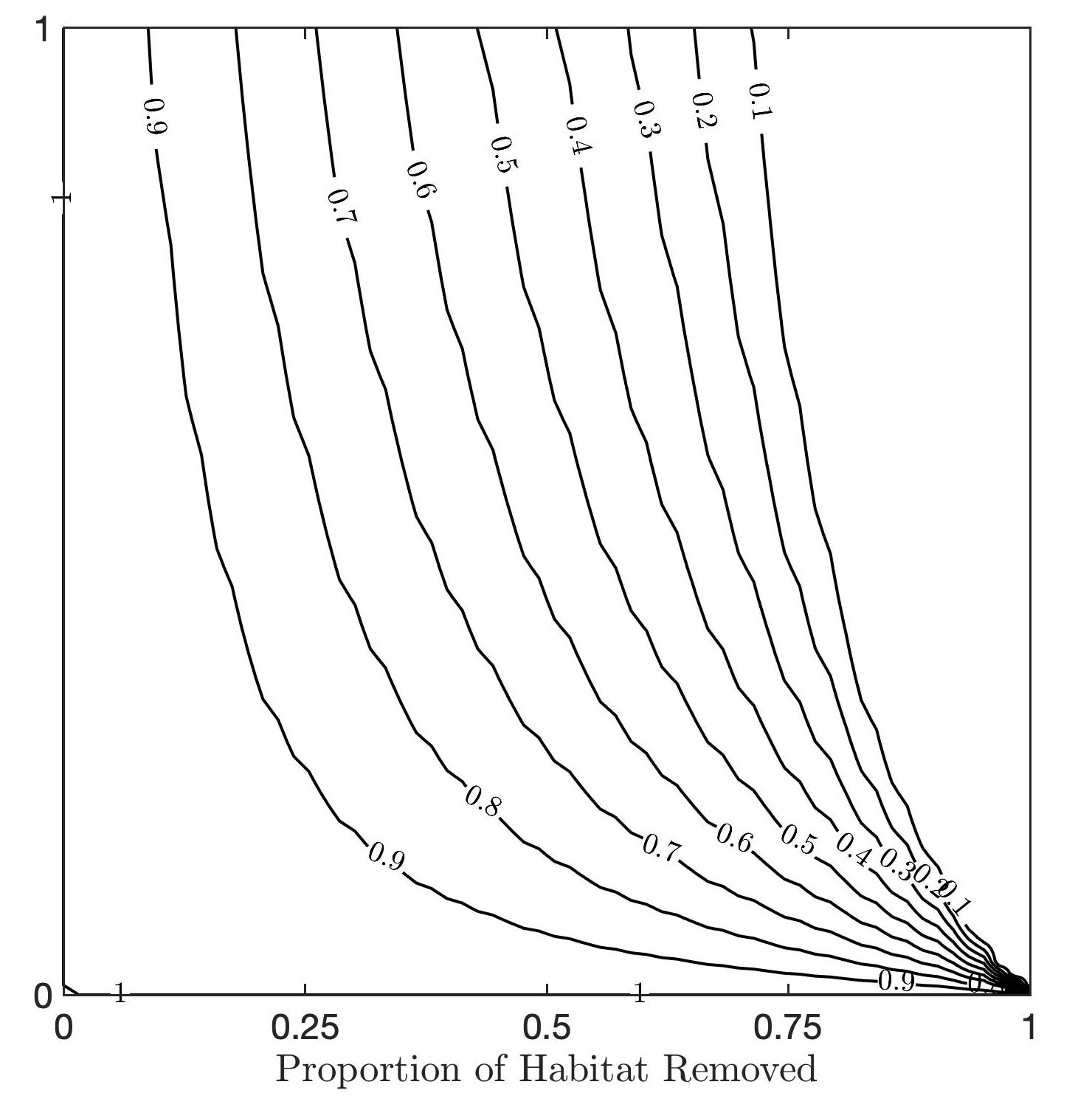}
  \caption{Diffusion fixed at $d=10.0$.}
\end{subfigure}\\
\begin{subfigure}{.475\textwidth}
  \centering
  \includegraphics[width=1.0\linewidth]{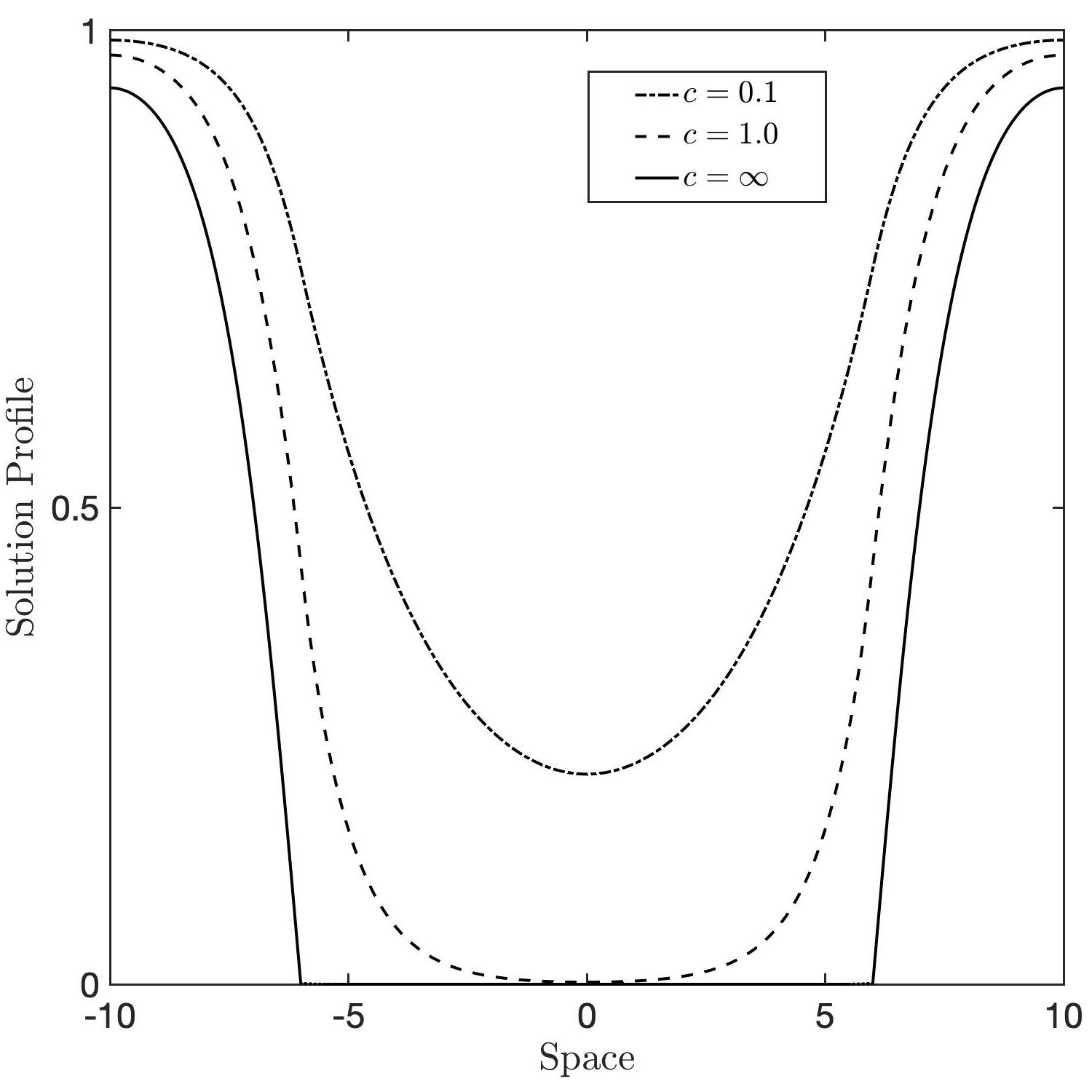}
  \caption{$60\%$ of habitat removed; no extinction threshold.}
\end{subfigure}%
\begin{subfigure}{.475\textwidth}
  \centering
  \includegraphics[width=1.0\linewidth]{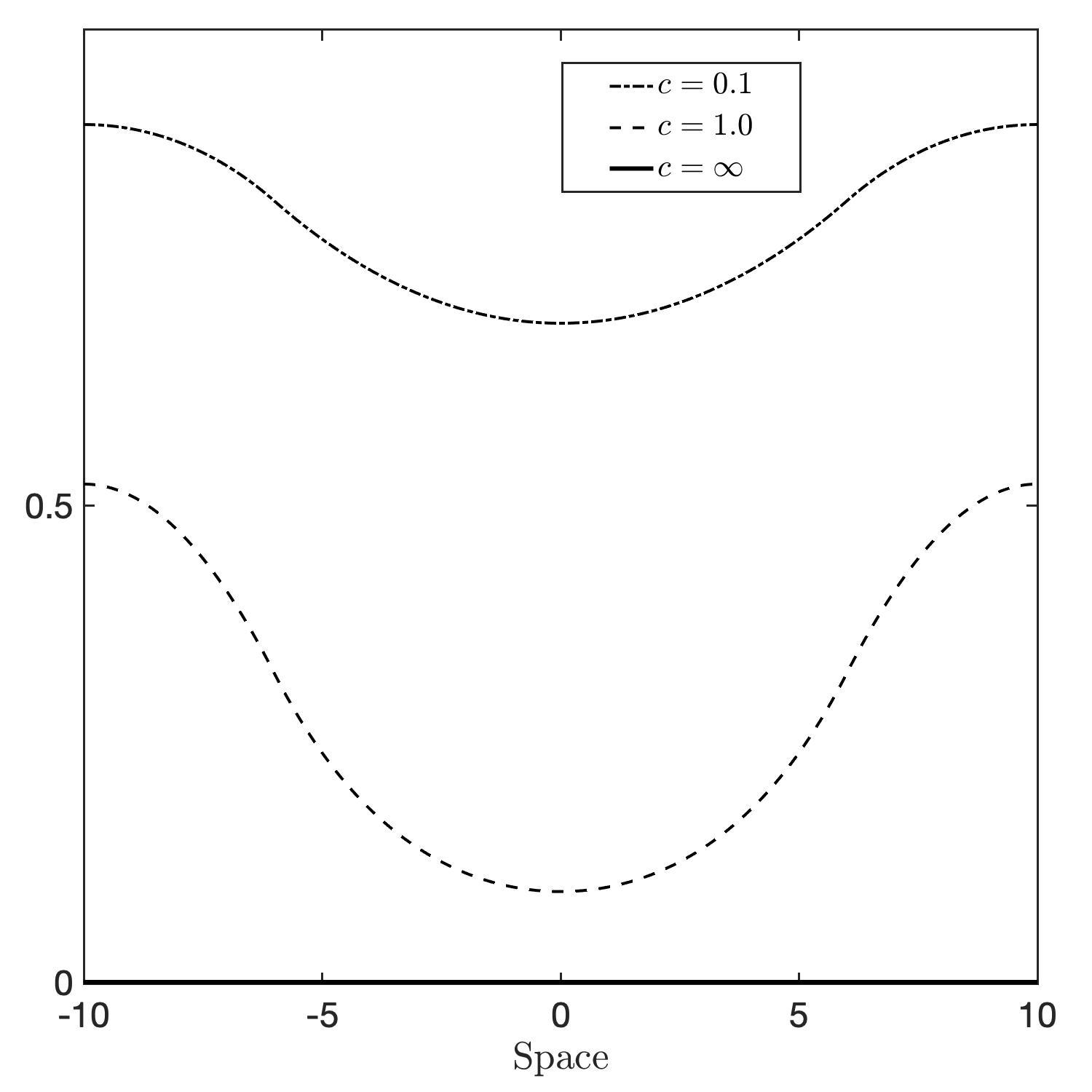}
  \caption{$60\%$ of habitat removed; extinction threshold $c_0 \approx 100$}
\end{subfigure}
\caption{Proportion of total population remaining at steady state (contour lines) after some proportion of the available habitat ($x$-axis) has been degraded at some level ($y$-axis). In this example, the domain is fixed to be $\Omega = (-10,10)$; growth in the undisturbed region is fixed at $1$. The degraded region is symmetric,  located at the centre of the domain.}
\label{fig:test}
\end{figure}

In Figure \ref{fig:test}, we illustrate our theoretical findings via numerical simulation. We consider a one dimensional landscape $\Omega = (-10,10)$ with a degraded interior $B = (-\delta,\delta)$ for $\delta \in (0,10)$. We assume logistic growth $f(x,u) = u(1-u)$ in the undisturbed region so that when $c=0$ (no environmental degradation), the population persists with density $1$ everywhere in $\Omega$. In panels (a)-(b), we plot the contour lines of the average population density at steady state (lying between $0$ and $1$) with respect to the proportion of habitat removed ($x$-axis) and the level of degradation ($y$-axis). Contour lines are drawn at levels where $10\%$ of the population has been lost. In panels (c)-(d), we plot the steady-state profile for $\delta = 6$ for different levels of degradation. In the left panels, we fix $d=1$; in the right panels we fix $d=10$, highlighting the influence of movement rates. First, we observe directly some expected monotonicity of the population levels (see Lemma \ref{monotonicallydecreasing-ss}). Panels (c)-(d) then demonstrate the existence versus non-existence of an extinction threshold. In panel (c), $\mu_{1,\infty} < 0$ and no extinction threshold is predicted; in pavel (d), $\mu_{1,\infty} > 0$ and there exists an extinction threshold $c_0 \approx 100$. It is interesting to observe the impact of movement rates of the population: when $d$ increases from panel (a) to panel (b), the contour lines are more concentrated in the centre. Therefore, the region for which the population can remain above $90\%$ (the left side) is larger for the faster population, and so the faster population is more resilient to introduction of moderately sized degradation regions. On the other hand, the regime of extinction is simultaneously increased (the right side), and so the faster population will experience extirpation sooner than the slower population should further habitat be removed.

%%%%%%%%%%%%%%%%%%%%%%%

\subsection{Preliminaries}\label{subsec-assumption}

We conclude the first section with the key assumptions used and some technical preparatory details. First we state the assumptions for $B$ and $f$.

\begin{assum}\label{assumptionA}
$B \Subset \O$ is an open subset with smooth boundary, comprised of finitely many disjoint components, each of which is simply connected.
\end{assum}

In practice, this corresponds to the ``cookie cutter" interpretation of habitat loss \cite{Pimm1995, Pimm2000}, suggesting that habitat loss is ``\textit{like a cookie cutter stamping out poorly mixed dough}". Geometrically, this ensures that the inner boundary $\p B$ does not touch the outer boundary $\p \O$, and prevents $B$ from breaking $\Omega$ into disjoint components in dimensions $N\geq2$.

\begin{assum}\label{assumptionf}
The function $f:(\overline{\Omega}\setminus B)\times[0,\infty)\to\mathbb{R}$ is assumed to satisfy the following conditions.
\begin{enumerate}[\rm(1)]
\item $f(\cdot,0) \equiv 0$, $f (\cdot, u)$ is H\oo lder continuous with exponent $\a \in (0,1)$ uniform with respect to $u$ in bounded sets;

\item $f(x,\cdot)\in C^{1}([0,\infty))$ for each $x \in \overline{\O} \setminus B$, $f_u(\cdot,0)$ is H\oo lder continuous with exponent $\a \in (0,1)$, and $f_u (x,0) > 0$ for some $x \in \O\setminus \overline{B}$;
\iffalse
\item For each $x\in \O \setminus \overline{B}$, $f(x,\cdot)$ is strongly subhomogeneous, namely,
$$
\l f(x, u) < f(x, \l u),\quad\forall \l \in (0,1)\,\,\text{and}\,\, u >0;
$$

\item There exists $N> 0$ such that $\sup_{x\in \overline{\O} \setminus B}f(x,u) < 0$ for all $u > N$.
\fi
\item For each $x\in \O \setminus \overline{B}$, $f(x,\cdot)$ is strictly concave down.
\end{enumerate}
\end{assum}

Items (1)-(2) ensure solutions are sufficiently smooth for our subsequent analysis. The positivity of $f_u (\cdot,0)$ somewhere in $\O \setminus\overline{B}$ is necessary to ensure that a nontrivial steady state may exist. Paired with the regularity of the domains $\O$ and $B$, we have regularity up to the boundary $\p \O$, and up to $\p B$ for problem \eqref{scalareqn-1}. Item (3) ensures uniqueness of the positive steady state (whenever it exists) and that the flow induced by the dynamical system is strongly monotone. This could be weakened to a subhomogeneity condition (see, e.g., \cite[Ch. 2.3]{Zhao2017}) when considering fixed values of $c$; we require concavity to obtain uniform convergence (in the variable $c$) from the time-dependent problem to the corresponding steady-state. A prototypical example satisfying Assumption \ref{assumptionf} is the heterogeneous logistic form $f(x,u) = u (m(x) - u)$ for some H\oo lder continuous function $m(x)$. For the remainder of the manuscript, Assumptions \ref{assumptionA} and \ref{assumptionf} are always assumed whenever $B$ and $f$ are involved.

We briefly introduce some important function spaces over the domain $\O \setminus \overline{B}$ and their relation to similar spaces over the entire domain $\O$. 
We denote by $C_B ^1 (\O)$ the collection of functions
$$
C_B^1 (\O) := \left\{ v \in C^1 (\overline{\O} \setminus B) : v \vert_{\p B} = 0 \right\},
$$
and consider $H_B^1 (\O)$ to be the closure of the space $C_B^1 (\O)$ with respect to the $H^1$-norm over $\O \setminus \overline{B}$. This way $H_B^1(\O)$ can be understood similar to the space $H_0 ^1(\O)$, with functions belonging to $H_B^1(\Omega)$ vanishing along $\p B$ in the trace sense. More precisely, it can be verified that
$$
H_{B}^{1}(\Omega)=\left\{u\in H^{1}(\Omega\setminus\overline{B}):\textbf{T}u=0\right\},
$$
where $\textbf{T}$ denotes the trace operator extending the notion of restricting a function on $\Omega\setminus B$ to $\partial B$.

For any $u\in C_{B}^{1}(\Omega)$ or $H_{B}^{1}(\Omega)$, we identify it with its zero extension into $B$. If $u\in H_{B}^{1}(\Omega)$, then $u \in H^{1}(\Omega)$. Conversely, if $u\in H^{1}(\Omega)$ with $u=0$ a.e. in $B$, then $u|_{\Omega}\in H_{B}^{1}(\Omega)$. Therefore, we identify $H_{B}^{1}(\Omega)$ with $\left\{u\in H^{1}(\Omega):u=0\,\,\text{a.e. in}\,\, B\right\}$, and simply write
\begin{equation*}
    \begin{split}
        H_{B}^{1}(\Omega)=\left\{u\in H^{1}(\Omega):u=0\,\,\text{a.e. in}\,\, B\right\}.
    \end{split}
\end{equation*}
We wish to emphasize that this identification rely crucially on Assumption \ref{assumptionA}; for less regular $B$, these two formulations of $H_B ^1(\Omega)$ need not agree \cite{Bonder2006,Bonder2007}.

We then have similar spaces with the temporal domain included. For $T>0$, we first set $Q_T: = \O \times (0,T)$ and $Q_{B,T}: = (\O \setminus \overline{B}) \times (0,T)$. We then define $H_B ^1 (Q_T)$ to be the closure of the set
$$
\left\{v\in C^{1}((\Omega\setminus B)\times[0,T]):v(\partial B\times[0,T])=0\right\}
$$ 
with respect to the $H^{1}$-norm over $(\Omega\setminus\overline{B})\times(0,T)$. By Assumption \ref{assumptionA} there holds
$$
H_{B}^{1}(Q_{T})=\left\{u\in H^{1}(Q_{B,T}):\textbf{T}u=0\right\}=\left\{u\in H^{1}(Q_{T}):u=0\,\,\text{a.e. in}\,\, B\times(0,T)\right\}.
$$

We organize the remainder of the paper as follows. In Section \ref{mainresults}, we study the connection between eigenvalue problems and prove Theorem \ref{convergence-thm-2}. Section \ref{apps} is devoted to investigating the connection between Cauchy problems \eqref{scalareqn-2} and \eqref{scalareqn-1}, and to proving Theorems \ref{convlem4} and \ref{finalscalarconvergence}. We discuss the implications of our results in Section \ref{sec-discussion}. Some proofs and well-known results concerning eigenvalue problems to be used throughout the paper are collected in the Supplementary Materials \ref{SM:proof_of_thm}-\ref{appendix-uniform-bounds}.

%%%%%%%%%%%%%%%%%%%%%%%%%%%%%%%%%%%%%%%%%%

\section{Connection between eigenvalue problems}\label{mainresults}

In this section we study the connection between problems \eqref{MainEig2-2} and \eqref{MainEig1-2}, and prove Theorem \ref{convergence-thm-2}. We then study the related sign-indefinite weight problems, providing a robust and complete picture of the connections between these problems. We refer the reader to the Supplementary Materials for more general consideration of eigenvalue problems with sign-indefinite weight and their connections to typical eigenvalue problems like \eqref{MainEig2-2} and \eqref{MainEig1-2}. We begin with a proof of Theorem \ref{convergence-thm-2}.

\begin{proof}[Proof of Theorem \ref{convergence-thm-2}]
Recall that $\mu_{1,c}$ has the variational characterization (see Proposition \ref{deglineig-app}):
\begin{equation}\label{variational-characterization-2022-04-02}
\mu_{1,c} = \inf_{\phi \in H ^1 (\O)} \left\{ \int_\O \left( d \magg{\grad \phi} - m_{c} \phi^2 \right)  \ : \ \int_\O \phi^2  = 1  \right\}.
\end{equation}
By Proposition \ref{deglineig-app} (ii),  we see that $\mu_{1,c}$ is strictly increasing in $c$. Since $\phi_{1,\i} \in H_B ^1 (\O)$, $\phi_{1,\infty}\in H^1 (\O)$ by zero extension in $B$. It follows from \eqref{variational-characterization-2022-04-02} and the normalization $\int_{\O} \phi_{1,\infty} ^2 = 1$ that
\eq{
\mu_{1,c} \leq \int_{\O} \left( d \magg{\grad \phi_{1,\i}} - m_{c} \phi_{1,\i}^2 \right) =\int_{\O\setminus\overline{B}} \left( d \magg{\grad \phi_{1,\i}} - f_u (\cdot,0) \phi_{1,\i}^2 \right) = \mu_{1,\infty},
}
where the second equality is a result of the eigen-equation satisfied by $\mu_{1,\infty}$ and $\phi_{1,\infty}$. Thus, $\mu_{1,c}$ is strictly increasing and uniformly bounded by $\mu_{1,\infty}$. Hence, $\mu_\i:=\lim_{c\to\infty}\mu_{1,c}$ exists and is finite. Obviously, $\mu_{\infty}\leq\mu_{1,\infty}$.

From the eigen-equation satisfied by $\mu_{1,c}$ and $\phi_{1,c}$ (or \eqref{variational-characterization-2022-04-02} with the understanding that the infimum is attained at $\phi_{1,c}$), there holds
\eq{
d \int_{\O} \magg{\grad \phi_{1,c}} &= \mu_{1,c} + \int_\O m_c \phi_{1,c}^2 \leq \mu_\infty + \norm{f_u (\cdot,0)}_{L^\i (\O \setminus \overline{B})},
}
where we have thrown away the negative term and used the normalization $\int_{\O} \phi_{1,c} ^2 = 1$. Hence, $\{ \phi_{1,c}\}_{c>0}$ is bounded in $H^1 (\O)$. Consequently, there exists a subsequence (still denoted by $\phi_{1,c}$) and some $\phi_{\infty}\in H^{1}(\Omega)$ such that 
\begin{equation}\label{convergence-2022-04-01}
    \lim_{c\to\infty}\phi_{1,c} = \phi_\i\quad\text{weakly in}\,\,H^1 (\O)\quad\text{and}\quad\text{strongly in}\,\, L^2 (\O).
\end{equation}

Note that
\eq{
c \int_{B} \phi_{1,c}^2 = \mu_{1,c} + \int_{\O \setminus\overline{B}} f_u (\cdot,0) \phi_{1,c}^2 - d \int_\O \magg{\grad{\phi_{1,c}}} \leq \mu_{\infty} + \norm{f_u (\cdot,0)}_{L^\i (\O\setminus\overline{B})},
}
leading to $\int_{B} \phi_{1,c}^2\leq \frac{1}{c}\left(\mu_\infty + \norm{f_u (\cdot,0)}_{L^\i (\O\setminus\overline{B})}\right)\to0$ as $c\to\infty$. This together with the strong convergence in \eqref{convergence-2022-04-01} implies $\int_{B} \phi_{1,\infty}^2=0$. Hence, $\phi_\i = 0$ a.e. in $B$, and so, $\phi_\i \in H_B ^1 (\O)$. Furthermore, since $\int_\O \phi_{1,c}^2 = 1$, the strong convergence in \eqref{convergence-2022-04-01} implies that $\int_\O \phi_\i^2 = 1$. Hence, $\phi_\i$ is nonzero and is a valid test function in the variational characterization of $\mu_{1,\infty}$. 

We now show that $\mu_{1,\infty} \leq \mu_\i$. Note that $\mu_{1,\infty}$ has the variational characterization (see Proposition \ref{deglineig-1-app}):
$$
\mu_{1, \infty} = \inf_{\phi \in H_{B} ^1 (\O)} \left\{ \int_{\O\setminus\overline{B}} \left( d \magg{\grad \phi} - f_u (\cdot,0) \phi^2 \right)  \ : \ \int_{\O\setminus\overline{B}} \phi^2  = 1  \right\}.
$$
This together with the weak lower semicontinuity of the norm $\|\cdot\|_{L^{2}(\Omega)}$ and \eqref{convergence-2022-04-01} leads to
\eq{
\mu_{1,\infty} &\leq \int_{\O\setminus \overline{B}} d\magg{\grad \phi_\i} - \int_{\O\setminus \overline{B}} f_u (\cdot,0) \phi_\i ^2 \\
&= \int_{\O} d\magg{\grad \phi_\i} - \int_{\O} m_{c} \phi_\i ^2 \\
&\leq \liminf_{c \to \i} \int_{\O} d\magg{\grad \phi_{1,c}} - \lim_{c \to \i} \int_{\O} m_{c} \phi_{1,c}^2= \liminf_{c \to \i} \mu_{1,c} = \mu_\i . 
%&\leq \liminf_{c \to \i} \left(\int_{\O} d\magg{\grad \phi_{1,c}} -  \int_{\O\setminus \overline{B}} f_u (\cdot,0)  \phi_{1,c}^2 + c \int_B \phi_{1,c}^2\right) \\
%&\leq \liminf_{c \to \i} \int_{\O} d\magg{\grad \phi_{1,c}} - \lim_{c \to \i}  \int_{\O} m_c  \phi_{1,c}^2  \\
%&= \liminf_{c \to \i} \left( \int_{\O} d \magg{\grad \phi_{1,c}} - m_c \phi_{1,c}^2 \right) = \liminf_{c \to \i} \mu_{1,c} = \mu_\i .
}
Hence, 
\begin{equation}\label{equality-2022-04-01}
    \mu_\i = \lim_{c \to \i} \mu_{1,c} = \mu_{1,\i}. 
\end{equation}
In particular, this implies that $\phi_\i$ solves the same eigenvalue problem as $\phi_{1,\i}$, and hence, $\phi_\i=\phi_{1,\i}$ by the uniqueness of the eigenfunction and the chosen normalization. 

It remains to show $\lim_{c\to\infty}\grad \phi_{1,c}=\grad \phi_\i$ in $L^{2}(\Omega)$ so that $\lim_{c\to\infty}\phi_{1,c}=\phi_{\infty}$ in $H^{1}(\Omega)$. Note that
\eq{
d  \int_\O \left( \magg{\grad \phi_{1,c}} - \magg{\grad \phi_\i} \right) &=  \mu_{1,c} - \mu_{1,\infty}  + \int_{\O \setminus \overline{B}} f_u (\cdot,0) ( \phi_{1,c}^2 - \phi_\i ^2 ) - c \int_B \phi_{1,c}^2 \\
&\leq  \mu_{1,c} - \mu_{1,\infty}  + \int_{\O \setminus \overline{B}} f_u (\cdot,0) ( \phi_{1,c}^2 - \phi_\i ^2 ).
}
Letting $c \to \i$ in the above inequality, we see from \eqref{equality-2022-04-01} and the strong convergence in \eqref{convergence-2022-04-01} that $\limsup_{c \to \i} \int_\O{\magg{\grad \phi_{1,c}}} \leq \int_\O \magg{\grad \phi_\i}$. As $\liminf_{c \to \i} \int_\O{\magg{\grad \phi_{1,c}}}\geq \int_\O \magg{\grad \phi_\i}$ due to the weak lower semicontinuity of the norm $\|\cdot\|_{L^{2}(\Omega)}$ and the weak convergence in \eqref{convergence-2022-04-01}, we find $\lim_{c \to \i} \int_\O{\magg{\grad \phi_{1,c}}}=\int_\O \magg{\grad \phi_\i}$, which together with the weak convergence in \eqref{convergence-2022-04-01} yields $\lim_{c\to\infty}\grad \phi_{1,c}=\grad \phi_{1,\i}$ in $L^{2}(\Omega)$.
\end{proof}

In the rest of this section, we explore further the eigenvalue problems with sign-indefinite weight associated with \eqref{MainEig2-2} and \eqref{MainEig1-2}, that is, 
\begin{equation}\label{MainEig2}
\begin{cases}
\D \psi + \l m_{c} \psi = 0,& \text{in}\quad \O ,  \\
\frac{\p \psi}{\p \nu } = 0, & \text{on} \quad\p \O,
\end{cases}
\end{equation}
and
\begin{equation}\label{MainEig1-1}
\begin{cases}
\D \psi +  \l f_u (\cdot,0)  \psi = 0 ,& \text{in} \quad\O \setminus \overline{B},  \\
\frac{\p \psi}{\p \nu } = 0 ,& \text{on}\quad\p \O , \\
\psi = 0 , &\text{on} \quad\p B.
\end{cases}
\end{equation}
While not directly related to the results obtained for the Cauchy problems in subsequent sections, it is necessary to also establish a connection between the principal eigenvalues to problems \eqref{MainEig2} and \eqref{MainEig1-1} in order to completely describe the relationship to problems \eqref{MainEig2-2} and \eqref{MainEig1-2}, particularly in the limiting case. We make this more precise following the statement of Theorem \ref{convergence-thm}.

It is easy to see that Assumption \ref{assumptionf} ensures that $m_{c}\in L^{\infty}(\Omega)$ is sign-changing and $f_u (\cdot,0) \in L^\i (\O \setminus \overline{B})$ is positive on a set of positive Lebesgue measure. Thus, Propositions \ref{MainEig2-thm-app} and \ref{existthm-app} apply to \eqref{MainEig2} and \eqref{MainEig1-1}, respectively. Set $c_*:= \frac{1}{\as{B}}\int_{\O \setminus B} f_u (\cdot,0)$. It is elementary to see that $\int_\O m_c < 0$ for all $c > c_{*}$. For each $c> c_{*}$, we denote by $\lambda_{1,c}$ the unique nonzero principal eigenvalue of \eqref{MainEig2}, and by $\psi_{1,c}$ the associated positive eigenfunction. Denote by $\lambda_{1,\infty}$ the unique positive principal eigenvalue of \eqref{MainEig1-1}, and by $\psi_{1,\infty}$ the associated positive eigenfunction.

We also have the following, connecting the principal eigenpairs of \eqref{MainEig2} and \eqref{MainEig1-1} as $c\to\infty$.

\begin{thm}\label{convergence-thm}
The following hold.
\begin{enumerate}[\rm(i)]
\item The function $c\mapsto \lambda_{1,c} $ is strictly increasing on $(c_{*},\infty)$, and $\lim_{c \to \i}\lambda_{1,c} =\lambda_{1,\infty}$.

\item $\lim_{c \to \i} \psi_{1,c} = \psi_{1,\infty}$ in $H ^1 (\O)$ under the normalization\\ $\int_\O m_c \psi_{1,c}^2 = \int_{\O \setminus B } f_u (\cdot,0) \psi_{1,\i} ^2 = 1$.
\end{enumerate}
\end{thm}

We can now make clear the important connection such a result endows. Proposition \ref{deglineig-app} (iii) provides a way of characterizing the sign of $\mu_{1,c}$ based on $\l_{1,c}$ and its relation to the size of $d$. Similarly, Proposition \ref{deglineig-1-app} (iv) gives an identical result for $\mu_{1,\i}$, whose sign is characterized by a relation between $\l_{1,\i}$ and $d$. Given that $\lim_{c\to\infty}\mu_{1,c}=\mu_{1,\i}$, it is expected that the quantity used to determine the sign of these eigenvalues should remain consistent. Theorem \ref{convergence-thm} confirms this is indeed the case, completing the analytical connection between these four problems. 

The biological insights gained by establishing this connection are also useful. Based on Propositions \ref{deglineig-app} and \ref{deglineig-1-app}, as is well known, a smaller rate of movement is favorable for a population's persistence in an environment with temporally static resources. How low movement must be to ensure persistence (or extinction, in the case of, e.g., undesirable pest populations) depends precisely on the size of the principal eigenvalues. With Theorem \ref{convergence-thm}, under increasing habitat degradation, the limiting case acts as a ``worst case scenario", providing a necessary and sufficient condition for the survival of species for any $c$. Alternatively, it provides insights into whether extinction is a possibility: if $d < \l_{1,\i}^{-1}$, extinction in the degradation case is impossible; otherwise, there exists $c$ large enough to guarantee extirpation.

\begin{proof}[Proof of Theorem \ref{convergence-thm}]
Since the proof is almost identical to the proof of Theorem \ref{convergence-thm-2}, we outline only the key steps. First, it is easy to deduce that $\lambda_{1,c}$ is strictly increasing and bounded above by $\lambda_{1,\infty}$. As a result, its limit exists and is given by the supremum, denoted by $\l_\i$. Then, we find that $\psi_{1,c}$ is uniformly bounded in $H^1 (\O)$ and thus has a convergent subsequence, weakly in $H^1 (\O)$ and strongly in $L^2(\O)$. Denote this by $\psi_\i$. Furthermore, $\psi_{1,c} \to 0$ a.e. in $B$, and so the candidate function $\psi_\i \in H_B ^1 (\O)$ as argued previously. One can show show that $\lambda_{1,\infty} \leq \l_\i$ by the weak lower semicontinuity of the norm. This implies that $\grad \psi_{1,c} \to \grad \psi_\i$ in norm, and hence the convergence is in fact strong. Uniqueness of the eigenfunction allows one to conclude that $\psi_\i = \psi_{1,\i}$.
\end{proof}

%%%%%%%%%%%%%%%%%%%%%%%%%%%%%%%%%

\section{Connection between Cauchy problems}\label{apps}

We now establish connection between the degradation problem \eqref{scalareqn-2} and the destruction problem \eqref{scalareqn-1} as given in Theorems \ref{convlem4} and \ref{finalscalarconvergence}. 

%%%%%%%%%%%%%%%%%%%%%%%%%%

\subsection{Global dynamics of degradation and destruction problems}

We first present results about the global well-posedness of the problems \eqref{scalareqn-2} and \eqref{scalareqn-1}. For \eqref{scalareqn-1}, the existence and uniqueness of a global classical solution 
\begin{equation}\label{gwp-destruction}
\begin{split}
&u_{\infty} \in C^{2+\a, 1+\a/2} ((\Omega\setminus\overline{B})\times(0,\infty)) \cap C^{+}( (\overline{\Omega}\setminus B)\times[0,\infty))\\
&\text{with}\,\,u_{\infty}(0,\cdot)\in H_B ^1 (\O) \cap C^{+}( \overline{\O}\setminus{B})
\end{split}
\end{equation}
follows from the classical local well-posedness theory (see e.g. \cite[Chapter 2]{Pao1992}) and the comparison principle.

For \eqref{scalareqn-2}, the existence and uniqueness of a global strong solution 
\begin{equation}\label{gwp-degradation}
    \begin{split}
        &u_c \in C^{\a} ( [0,\i); C^+ ( \overline{\O})) \cap C^1 ( (0,\i); C^+ (\overline{\O}))\\
        &\text{with}\,\,u_{c}(0,\cdot)\in C^{+}( \overline{\O})
    \end{split}
\end{equation}
follows from the proof of \cite[Theorem 2.1]{Salmaniw2022} (with proper modifications to approximate $\Ind_{\O \setminus B}f(x,u)$) and the comparison principle. While the solution $u_c$ is not twice continuously differentiable in space due to the discontinuity in the right hand side of \eqref{scalareqn-2}, $u_c \in C^{1 + \a, (1+\a)/2} (Q_T)$ for any $\a \in (0,1)$ and any $T>0$ by the Sobolev embedding. It is important to point out that these regularity results hold for each $c>0$ fixed, but \textit{do not a priori hold independent of $c$}. 

In what follows, $u_{c}$ and $u_{\infty}$ are respectively understood in the sense of \eqref{gwp-degradation} and \eqref{gwp-destruction} unless otherwise specified. The global dynamics of the degradation problem \eqref{scalareqn-2} is given in the following.

\begin{thm}\label{globaldegradation}
The following hold for each $c\in(0,\infty)$.
\begin{enumerate}[\rm(i)]
\item If $\mu_{1,c} < 0$, then \eqref{scalareqn-2} admits a unique positive steady state $u_c^*\in W^{1,p}(\Omega)$ for any $p\geq1$, and $u_c(\cdot,t) \to u_c^*$ in $C( \overline{\O})$ as $t \to \i$ whenever $u_{c}(0,\cdot)\not\equiv 0$;

\item If $\mu_{1,c} \geq 0$, then $0$ is the only steady-state solution to \eqref{scalareqn-2} and $u_c(\cdot,t) \to 0$ in $C(\overline{\O})$ as $t \to \i$.
\end{enumerate}
Moreover, 
\begin{itemize}
\item if $\mu_{1,\i} \leq0$, then $\rm(i)$ holds for any $c \in (0,\i)$;
\item if $\mu_{1,\i} > 0$, then there exists $c^*>0$ such that $\rm(i)$ holds for all $c\in(0,c^{*})$ and $\rm(ii)$ holds for all $c \in (c^*,\i)$.
\end{itemize}
\end{thm}
\begin{proof}
This result follows from essentially from classical results, see e.g. \cite[Ch. 4, Theorem 4.1]{Ni2001}, alongside some of the arguments made in \cite{Salmaniw2022} since the function $\Ind_{\O \setminus B}f(x,u) - c \Ind_{B} u$ may be discontinuous along $\p B$. More precisely, a strong monotonicity result holds for problem \eqref{scalareqn-2}, and the uniqueness of the positive steady state follows from the concavity of the function $f(x,u) - c \Ind_{B} u$ (see Proposition 2.2 and Theorem 2.1 in \cite{Salmaniw2022}). The result then follows from the theory of monotone flows. The ``Moreover" part follows directly from Theorem \ref{convergence-thm-2}(1).
\end{proof}

\begin{remark}
    Whenever it exists, the unique quantity $c^*$ is exactly the extinction threshold $c_0$ given in Corollary \ref{cor:exthresh}. This is the precise value at which any further habitat degradation (for a fixed configuration $B$) results in deterministic extirpation. Panels (c)-(d) of Figure \ref{fig:test} demonstrate this point: in panel (d) there exists an extinction threshold, while in panel (c) there does not.
\end{remark}

The global dynamics of the destruction problem \eqref{scalareqn-1} are summarized in the next result.

\begin{thm}\label{globaldestruction}
Assume $0 \lneq u_\infty(0,\cdot) \in C_B ^1 (\overline{\O})$. The following hold.
\begin{enumerate}[\rm(i)]
\item If $\mu_{1,\infty} < 0$, then \eqref{scalareqn-1} admits a unique positive steady state $u^*_{\infty}\in C^{2+\a} (\O \setminus \overline{B}) \cap C( \overline{\O} \setminus B)$, and $u_{\infty} \to u^*_{\infty}$ in $C( \overline{\O} \setminus B)$ as $t \to \i$; 

\item If $\mu_{1,\i} \geq 0$, then $0$ is the only steady state to \eqref{scalareqn-1}, and $u_{\infty} \to 0$ in $C(\overline{\O} \setminus B)$ as $t \to \i$.
\end{enumerate}
\end{thm}
\begin{proof}
This result follows from similar arguments made in the proof of Theorem \ref{globaldegradation}, however there are some technical issues to address due to the Dirichlet boundary condition along the boundary $\p B$. To address this, we set $X := C_B ^1 (\overline{\O})$ and recall the strong partial order on $X$ generated by the cone $X^{+}=\left\{v\in X:v\geq0\right\}$ with interior
\eq{
X^{++} = \left\{ v \in X : v > 0 \text{ in } \overline{\O} \setminus \overline{B} \text{ and } \frac{\p v}{\p \nu } < 0 \text{ on } \p B \right\} .
}
The global existence, uniqueness and regularity of solutions to \eqref{scalareqn-1} with initial data in $X^{+}$, paired with the comparison principle, ensure that \eqref{scalareqn-1} generates a strongly monotone flow on $X^{+}$. 

When $\mu_{1,\i} < 0$, the existence of a positive steady state $u^* \in X^{++}$ follows from a sub/super solution argument, where any sufficiently large constant acts as a super solution due to Assumption \ref{assumptionf} (3), and $w = \e \tilde \psi_1$ acts as a sub solution since we may choose $\e$ sufficiently small but positive so that $f_u (\cdot,0) - \frac{f(\cdot,w)}{w} \leq - \mu_1(\i)$, again by Assumption \ref{assumptionf} (3). Furthermore, the concavity ensures that the steady state is unique (subhomogeneity is sufficient, see \cite[Ch. 2.3]{Zhao2017}). Since problem \eqref{scalareqn-1} generates a strongly monotone flow in $X^{+}$, we then conclude that $u^*$ is globally attractive for any initial data $u_0\in X^{+} \setminus \{ 0 \}$. 

When $\mu_{1,\infty}\leq0$, $0$ is the only steady state solving problem \eqref{scalareqn-1} since $w = \e \tilde\psi_1$ is a (strict) super solution for all $\e>0$ by Assumption \ref{assumptionf} (3), and $0$ is globally attractive.
\end{proof}

%%%%%%%%%%%%%%%%%%%%%%%%%%%%%%%%%%%%

\subsection{From habitat degradation to habitat destruction: steady states}

This subsection is devoted to the convergence between the steady states $u_c^*$ and $u^*_{\infty}$ as $c \to \i$ as stated in Theorem \ref{convlem4}. We need the following lemma.

\begin{lem}\label{monotonicallydecreasing-ss}
Assume $\mu_{1,\infty}<0$. Then, for any $0 < \underline{c} < \overline{c}$, there holds $u_{\infty}^{*}\leq u_{\overline{c}}^* < u_{\underline{c}}^*$ in $\overline{\Omega}$.
%    \item $u^*_{\infty} \leq \lim_{c\to\infty}u_c^*$ in $\overline{\O}$.
\end{lem}
\begin{proof}
Since $\mu_{1,\i} < 0$, $u_{c}^{*}$ exists for all $c>0$. Then, it is easy to see that $u_{\overline{c}}^* \leq u_{\underline{c}}^*$ by the concavity of $f(x,\cdot)$ and the strong maximum principle for strong solutions, see the proof of e.g. \cite[Proposition 2.2]{Salmaniw2022}. The the strong maximum principle and Hopf's lemma implies that either $u_{\overline{c}}^* < u_{\underline{c}}^*$ or $u_{\overline{c}}^* \equiv u_{\underline{c}}^*$ in $\overline{\O}$. By the uniqueness of the steady state solution, the second cannot hold, and the strict inequality follows.

Let $c>0$. Note that both $u_{c}^{*}$ and $u_{\infty}^{*}$ satisfy 
\begin{equation*}
    \begin{cases}
0 = d \D u + f(x,u), & \text{in}\quad \O \setminus \overline{B}, \\
\frac{\p u}{\p \nu } = 0, & \text{on}\quad  \p \O.
\end{cases}
\end{equation*}
As $u_\i^* = 0 < u_c ^*$ on $\p B$, we apply the comparison principle to conclude $u_\i^* \leq u_c ^*$ in $\overline{\O} \setminus B$. Since $u_\i^*$ is identified with its extension by zero in $B$, we automatically have that $u_\i^{*} < u_c ^*$ in $\overline{B}$ by the positivity of $u_c^*$. Hence, $u_\i^*  \leq u_c ^*$ in $\overline{\O}$. 
\end{proof}

We now establish Theorem \ref{convlem4}. Its proof is instructive for the more difficult parabolic analog (i.e. Theorem \ref{finalscalarconvergence}) as we require fewer estimates to conclude our desired result.

\begin{proof}[Proof of Theorem \ref{convlem4}]

The existence and uniqueness of positive steady states follow from Theorems \ref{globaldegradation} and \ref{globaldestruction}. It remains to show the convergence result.

Lemma \ref{monotonicallydecreasing-ss} asserts that $\{u_c ^*\}_{c\gg1}$ is a decreasing sequence of functions bounded below by $u_\i^*$. Hence, the pointwise limit $u^{*}:=\lim_{c\to\infty}u_c ^*$ exists in $\overline{\Omega}$ and is nontrivial. This is our candidate solution to the limiting problem. 

Multiplying the equation satisfied by the steady state $u_c ^*$ by itself, integrating over $\O$ and integrating by parts, we obtain
\begin{equation}\label{eqn-2022-04-03}
    d \int_\O \magg{\grad u_c ^*} = \int_{\O \setminus B} f(\cdot,u_c^*)u_c^* - c \int_B (u_c^*)^2 \leq \norm{f_u (\cdot,0)}_{L^\i (\O\setminus\overline{B})} \norm{u_c^*}_{L^2 (\O)}^2,
\end{equation}
where used Assumption \ref{assumptionf} (3) in the inequality. Hence, $\{u_c^*\}_{c\gg1}$ is bounded in $H^1 (\O)$. Consequently, there exists a subsequence, still denoted by $u_{c}^{*}$, such that 
\begin{equation}\label{convergence-2022-04-03}
    \lim_{c\to\infty}u_c ^* = u^*\quad\text{strongly in}\,\, L^2 (\O)\quad\text{and}\quad \text{weakly in}\,\, H^1 (\O).
\end{equation}
In particular, for any $\phi \in H_B^1 (\O)$ (considered as an element in $H^{1}(\Omega)$ after  zero extension in $B$), we have $- c \int_{B} u_c ^* \phi = 0$ for all $c\gg1$, and
\eq{
\lim_{c\to\infty}d \int_\O \grad u_c ^* \cdot \grad \phi = d \int_\O \grad u^* \cdot \grad \phi , \quad\lim_{c\to\infty}\int_\O f(\cdot,u_c ^*) \phi = \int_\O f(\cdot, u^*) \phi.
}
Therefore, $u^*$ satisfies $-d \D u^*  = f(x,u^*)$ in $\O\setminus\overline{B}$ in the weak sense. 

We now show that $u^{*}\in H_{B}^{1}(\Omega)$. We see from the equality in \eqref{eqn-2022-04-03} that
\eq{
c \int_{B} (u_c ^*)^2 &\leq \norm{f_u (\cdot,0)}_{L^\i (\O\setminus\overline{B})} \int_\O (u_c ^*)^2.
}
As $\sup_{c\gg1}\int_\O (u_c ^*)^2<\infty$, we arrive at $\lim_{c\to\infty}\int_{B} (u_c ^*)^2=0$. It then follows from the convergence in \eqref{convergence-2022-04-03} or the monotone convergence theorem that $\int_{B}(u^{*})^{2}=0$, and hence, $u^*=0$ a.e. in $B$. In particular, $u^* \in H_B^1(\O)$. 

Combining these results, we conclude from the elliptic regularity theory that $u^*$ is a steady state to \eqref{scalareqn-1}, and therefore, $u^{*}=u_\i ^*$ by the uniqueness of positive steady states. As $u^{*}\in C(\overline{\Omega})$, we conclude the result from Dini's theorem.
\end{proof}

\begin{remark}
We cannot expect a stronger notion of convergence over the entire domain $\O$ in a classical sense than what was shown above. Informally, this can be made intuitive if one considers the fact that $\frac{\p u^*_{\infty}}{\p \nu }$ is negative along $\p B$ while $u^*_{\infty}$ is identically zero inside of $B$. Hence, we expect the classical derivative of $u^*_{\infty}$ to be discontinuous along $\p B$. However, stronger notions of convergence are readily established away from the boundary of $B$ through the usual arguments.
\end{remark}

%%%%%%%%%%%%%%%%%%%%%%%%%%%%%%%%%%

\subsection{From habitat degradation to habitat destruction: general solutions}

This subsection is devoted to the proof of Theorem \ref{finalscalarconvergence}. We prove several lemmas before doing so. The following addresses the monotonicity of solutions in $c$.

\begin{lem}\label{monotonicallydecreasing}
Assume $0 < \underline{c} < \overline{c}$. If $u_{\underline{c}}(\cdot,0)=u_{\overline{c}}(\cdot,0)\in C^{+}(\overline{\Omega})\setminus\{0\}$, then $u_{\underline{c}}>u_{\overline{c}}>0$ in $\overline{\Omega}\times(0,\infty)$.
\end{lem}
\begin{proof}
Set $w :=u_{\overline{c}} - u_{\underline{c}}$ and $w^+ := \max \{ 0, w\}$. Note that $u_{\overline{c}}$ and $u_{\underline{c}}$ are bounded. This together with the regularity assumption on $f$ implies the existence of some $K>0$ such that
\eq{
\frac{1}{2} \frac{d}{dt} \int_\O ( w^+) ^2 &= - d \int_\O \magg{\grad w^+} + \int_\O w^+ ( f(\cdot,u_{\overline{c}}) - f(\cdot,u_{\underline{c}}) )\leq K \int_\O (w^+)^2.
}
Gr{\"o}nwall's inequality implies that $w^+ = 0$ a.e. in $Q_T$ and hence $u_{\overline{c}} \leq u_{\underline{c}}$ holds in all of $Q_T$ by the smoothness of the solutions. Then, since $u_c \in C^1 ( (0,\i); C^+ (\overline{\O}))$, the strong maximum principle for strong solutions applies, see e.g. \cite{arena}. Indeed, if there exists a point $(x_0,t_0) \in Q_T$ such that $w = 0$, it follows that $w \equiv 0$ in $\O$ for all $t \in (0,t_0)$, a contradiction to the uniqueness of solutions. Finally, if there exists a point $x_0 \in \p \O$ such that $w(x_0, t_0) = 0$ for some $t_0 > 0$, Hopf's lemma implies that $\frac{\p w}{\p \nu } (x_0,t_0) > 0$, a contradiction to the homogeneous Neumann boundary condition satisfied by $w$ along $\p \O$. Hence, $w<0 \Rightarrow u_{\overline{c}} < u_{\underline{c}}$ in $\overline{\O}\times(0,\infty)$. The conclusion $u_{\overline{c}}>0$ in $\overline{\O}\times(0,\infty)$ follows from similar arguments.
\end{proof}

As an immediate consequence of Theorem \ref{globaldegradation} and Lemma \ref{monotonicallydecreasing}, we have the following result.

\begin{cor}\label{cor-uniform-converge-to-0}
Suppose $\mu_{1,\i} > 0$. Then, for any initial data $u_c (0,\cdot) \in C^+ (\Omega) \setminus \{ 0 \}$ independent of $c\gg1$, there holds
$$
\lim_{t \to \i} \sup_{c \gg 1} \norm{u_c (\cdot,t)}_{C(\overline{\O})} = 0.
$$
In fact, for any $\tilde c \in (c^*,\infty)$ there exist $r = r(\tilde c)>0$ and $M>0$ (depending only on $\tilde{c}$ and the common initial data) such that
$$
\sup_{c \geq \tilde c}\norm{u_c(\cdot,t)}_{C(\overline{\Omega})} \leq M e^{-r t}, \quad \forall\, t \geq0.
$$
\end{cor}
\begin{proof}
Let $c^{*}$ be as in Theorem \ref{globaldegradation}. Then, $\mu_{1,c} > 0$ for all $c > c^*$. Fix $\tilde c > c^*$ and set $W := M e^{-\mu_{1,\tilde c} t} \phi_{1,\tilde c}$, where $\phi_{1,\tilde c}$ is the positive eigenfunction to problem \eqref{MainEig2-2} for $m_{\tilde c}$. It is easy to see that $W$ is a super solution to $u_{\tilde c}$ whenever $M$ is chosen large enough that $u_{\tilde c}(\cdot,0) \leq M \phi_{1,\tilde{c}}$. One immediately has that $\lim_{t\to \i}\norm{u_{\tilde c} (\cdot,t)}_{C(\overline{\O})} \leq M \max_{\Omega}\phi_{\tilde c} \lim_{t \to \infty} e^{- \mu_{1,\tilde c} t} =  0$. The result then follows from Lemma \ref{monotonicallydecreasing}.
\end{proof}

The next result addresses the uniform convergence over finite time intervals. As the arguments are technical but somewhat standard, we include the proof in the Supplementary Materials \ref{SM:proof_of_thm}.

\begin{lem}\label{convcor1}
If $u_{c}(\cdot,0)=u_{\infty}(\cdot,0) \in H_B^1 (\O) \cap C^{+}(\overline{\O})$ for all $c\gg1$ and \\ $\supp(u_{\infty}(\cdot,0)) \Subset \overline{\O} \setminus \overline{B}$, then for each $T>0$, 
$$
\lim_{c \to \i}u_c = u_{\infty}\,\,\text{uniformly in}\,\, \overline{\O } \times [0,T].
$$
\end{lem}

%% proof removed here
%%%

Next we treat semi-infinite time intervals $(T,\infty)$.

\begin{lem}\label{uniconvscalar}
Assume $\mu_{1,\i} < 0$. If $u_{c}(\cdot,0)=u_{\infty}(\cdot,0) \in H_B^1 (\O) \cap C^{+}(\overline{\O})$ for all $c\gg1$ and $\supp(u_{\infty}(\cdot,0)) \Subset \overline{\O} \setminus \overline{B}$, then there exist $r>0$ and $M=M(u_{\infty})>0$ such that
\eq{
\norm{u_c(\cdot,t) - u_c ^*}_{C( \overline{\O})}\leq M e^{- rt} + \norm{u^*_{\infty} - u_c ^*}_{C( \overline{\O})},\quad \forall t>0\,\,\text{and}\,\,c\gg1.
}
\end{lem}
\begin{proof}
The conclusion of the lemma follows from the following two steps. Denote by $u_{0} := u_c (\cdot, 0)=u_{\infty}$ the common initial data. 

\paragraph{\bf Step 1} We show the existence of $r_{1}>0$ and $M_{1}=M_{1}(u_{0})>0$ such that 
$$
u_c(\cdot,t)-u^*_c \leq  M_1e^{-r_{1}t},\quad\forall t\geq0\,\,\text{and}\,\, c\gg1.
$$

By Theorem \ref{globaldegradation}, $u_c^*$ exists for all $c>0$.  Denote by $\mu_1 (d,\Ind_{\O \setminus B} f_u (\cdot,u_c^*) - c \Ind_{B})$ the principal eigenvalue of \eqref{MainEig2-2-app} with $h = \Ind_{\O \setminus B} f_u (\cdot,u_c^*) - c \Ind_{B}$, and by $\hat{\psi}_{c}$ the associated positive eigenfunction satisfying the normalization $\int_{\Omega}\hat{\psi}_{c}^{2}$=1. Notice that $\mu_1 (d,\Ind_{\O \setminus B} f_u (\cdot,u_c^*) - c \Ind_{B}) > 0$ for any $c \in (0,\i)$ due to the concavity of $f(x,\cdot)$. We claim that 
\begin{equation}\label{claim-July-05-2022}
    \liminf_{c\to\infty}\mu_1 (d,\Ind_{\O \setminus B} f_u (\cdot,u_c^*) - c \Ind_{B}) > 0.
\end{equation}

Denote by $\mu_1 (\Ind_{\O \setminus B} f_u (\cdot,u_\i ^*) - c \Ind_B)$ the principal eigenvalue of \eqref{MainEig2-2-app} with $h = \Ind_{\O \setminus B} f_u (\cdot,u_\i ^*) - c \Ind_B$. By a minor modification of the proof of Theorem \ref{convergence-thm-2}, it is not difficult to find that 
\begin{equation}\label{a-limit-July-5-2022}
\lim_{c \to \i}\mu_1 (d,\Ind_{\O \setminus B} f_u (\cdot,u_\i ^*) - c \Ind_B) = \mu_1 (d,\Ind_{\O \setminus \overline{B}}f_u (\cdot,u_\i ^*), B),
\end{equation}
where $\mu_1 (d,\Ind_{\O \setminus \overline{B}}f_u (\cdot,u_\i ^*), B)$ is the principal eigenvalue of \eqref{MainEig1-2-app} with\\ $m=\Ind_{\O \setminus \overline{B}} f_u (\cdot,u_\i ^*)$.

By the variational characterization of $\mu_1 (d,\Ind_{\O \setminus B} f_u (\cdot,u_\i ^*)- c \Ind_B)$, we find
\eq{
\mu_1 &(d,\Ind_{\O \setminus B} f_u (\cdot,u_\i ^*) - c \Ind_B) = \inf_{\phi \in H^1 (\O)} {\scriptstyle \left\{ d \int_\O \magg{\grad \phi} - \int_\O \phi^2 (\Ind_{\O \setminus B} f_u (\cdot,u_\i ^*) - c \Ind_B) \ : \ \int_\O \phi^2 = 1  \right\} } \\
&\leq d \int_\O \magg{\grad \hat\psi_c} - \int_\O \hat\psi_c ^2 ( \Ind_{\O \setminus B} f_u (\cdot,u_\i ^*) - c \Ind_B ) \\
&= \mu_1 (d,\Ind_{\O \setminus B} f_u (\cdot,u_c ^*) - c \Ind_B) + \int_{\O\setminus B} \hat\psi_c ^2 ( f_u (\cdot,u_c ^*) - f_u (\cdot,u_\i^*) ).
}
Theorem \ref{convlem4} and the normalization $\int_{\Omega}\hat{\psi}_{c}^{2}$=1 imply that $\lim_{c\to\infty}\int_{\O\setminus B} \hat\psi_c ^2 ( f_u (\cdot,u_c ^*) - f_u (\cdot,u_\i^*) )=0$. It then follows from \eqref{a-limit-July-5-2022} that
\eq{
 \mu_1 (d,\Ind_{\O \setminus \overline{B}}f_u (\cdot,u_\i ^*),B)&=\liminf_{c \to \i}\mu_1 (d,\Ind_{\O \setminus B} f_u (\cdot,u_\i ^*) - c \Ind_B) \nonumber \\
 &\leq \liminf_{c \to \i} \mu_1 (d,\Ind_{\O \setminus B} f_u (\cdot,u_c^*) - c \Ind_B).
}
Since $\mu_1 (d,\Ind_{\O \setminus \overline{B}}f_u (\cdot,u_\i^*),B) > 0$, the claim \eqref{claim-July-05-2022} follows.

%%%%%%%%%%%%%%%%%%%%%%%%%%%%%%%%%%%%%

Since $u_{0}$ is continuous and compactly supported in $\overline{\Omega}\setminus\overline{B}$, and $\{\hat{\psi}_{c}\}_{c\gg1}$ is locally uniformly positive in $\overline{\Omega}\setminus\overline{B}$ by Lemma \ref{eig-lowerbound}, there exists $\tilde{M}_{1}>0$ such that $u_{0}\leq \tilde{M}_{1}\hat{\psi}_{c}$ for all $c\gg1$. Set
$$
W_{c} := \tilde{M}_1 e^{- \mu_1 (\Ind_{\O \setminus B} f_u (\cdot,u_c^*) - c \Ind_{B}) t} \hat{\psi}_{c}.
$$
It is straightforward to check that $W_{c}$ satisfies
\eq{
(W_{c})_t - d \D W_{c} &=  ( \Ind_{\O \setminus B} f_u (\cdot,u_c^*) - c \Ind_{B} )W_{c} \quad\text{in}\quad\Omega\times(0,\infty).
}
Note that in $\Omega\times(0,\infty)$, $w_{c}:= u_c - u^*_c$ obeys
\eq{
(w_{c})_t - d \D w = \Ind_{\O \setminus B} ( f(\cdot,u_c) - f(\cdot, u_c^*) ) - c \Ind_{B} w_{c} \leq  ( \Ind_{\O \setminus B} f_u (\cdot,u_c^*) - c \Ind _B)w_{c},
}
where we used the concavity of $f$ in the inequality. Obviously, both $W_{c}$ and $w_{c}$ satisfy the homogeneous Neumann boundary condition on $\partial\Omega$. Since $w_{c}(\cdot,0) = u_{0} - u_c ^* \leq \tilde{M}_1 \hat\psi_c = W_{c}(\cdot,0)$, we apply the comparison principle to arrive at $w_{c} \leq W_{c}$. Note that Lemma \ref{eig-upperbound} yields $\sup_{c\gg1}\sup_{\Omega}\hat{\psi}_{c}<\infty$. Hence, setting $r_{1}: = \liminf_{c \to\infty} \{ \mu_1 (f_u (\cdot,u_c^*) - c \Ind_{B} \} -\delta> 0$ for some fixed $0<\delta\ll1$ and $M_{1}:=\tilde{M}_{1}\sup_{c\gg1}\sup_{\Omega}\hat{\psi}_{c}+1$, we find
$u_c(\cdot,t) - u_c ^* \leq M_1 e^{- r_{1}t}$ for all $t\geq0$ and $c\gg1$.

\medskip

\paragraph{\bf Step 2} We show the existence of $r_2 > 0$ and $M_2 = M_2(u_{0})>0$ such that 
$$ 
u_c^* - u_c(\cdot,t) \leq M_2e^{- r_{2} t} + \norm{u_\i^* - u^*_c}_{C(\overline{\O})}, \quad \forall t \geq 0 \text{ and } c \gg 1.
$$

As we are treating the lower bound for $u_{c}$, we may assume without loss of generality that $u_0 \leq u_ \i ^*$. Note that Lemma \ref{monotonicallydecreasing} ensures that $u_\i \leq u_c$ for all $c \gg 1$, leading to
\eql{
u_c ^* - u_c(\cdot,t) \leq \norm{u^*_\i - u_c ^*}_{C(\overline{\Omega})} + u ^*_\i - u_\i(\cdot,t), \quad \forall t \geq 0 \text{ and } c \gg 1,
}{ineq1.9999}
where $u_\i$ solves \eqref{scalareqn-1} with initial data $u_0$. Hence, it suffices to derive an exponential-in-time upper bound for  $u ^*_\i - u_\i(\cdot,t)$.

We claim that there exist $t_{0}\gg1$ and $v\in L^\i(\O)$ such that
\begin{equation}\label{claim-2022-09-17}
0\lneqq v\leq u_{\infty}(\cdot,t)\quad\text{in}\quad \Omega\setminus\overline{B},\quad\forall t\geq t_{0}.    
\end{equation}
Indeed, since Theorem \ref{globaldestruction} ensures that $u_\i(\cdot,t) \to u_\i ^*$ uniformly in $\overline{\O}\setminus B$ as $t \to \i$, for some fixed $V \Subset \O \setminus \overline{B}$ there is $t_{0}\gg1$ such that $\inf_{V\times(t_{0},\infty)} u_\i>0$. The claim follows readily.

Set $F(u):=\frac{f(\cdot,u_\i ^*) - f(\cdot,u)}{u_\i ^* - u}$. We show $r_{2}:=\mu_1 (d, F(v), B))>0$. Indeed, since $u_{\infty}(\cdot,t)\leq u_{\infty}^{*}$ for all $t\geq0$ by the choice of the initial data $u_{0}$, we find $v\leq u_{\infty}^{*}$ from \eqref{claim-2022-09-17}. It follows from the concavity of $f$ that $F(v)\leq F(0)$. Noticing that $\mu_{1}(d,F(0),B)=0$ (as $u_{\infty}^{*}$ is exactly the associated eigenfunction), we deduce from Lemma \ref{deglineig-1-app} (ii) and $v\gneqq0$ (by \eqref{claim-2022-09-17}) that $r_{2}=\mu_1 (d, F(v),B) > \mu_1 (d, F(0),B ) = 0$.

Set $W_{\infty}:=M_{2}e^{-r_{2}(t-t_{0})}\hat{\psi}_{\infty}$ in $\Omega\setminus\overline{B}\times[t_{0},\infty)$, where $\hat{\psi}_{\infty}$ is the positive eigenfunction of \eqref{MainEig1-2-app} with $m = F(v)$ associated with the principal eigenvalue $r_{2}=\mu_1 (d, F(v),B)$, and $M_{2}>0$ is such that $u_{\infty}^{*}\leq M_{2}\hat{\psi}_{\infty}$. Such a $M_{2}$ exists due to the positivity of $u_{\infty}^{*}$ and $\hat{\psi}_{\infty}$ in $\overline{\Omega}\setminus\overline{B}$ and the negativity of the outer normal derivative of $u_{\infty}^{*}$ and $\hat{\psi}_{\infty}$ along $\partial B$. It is straightforward to check that $W_\i$ satisfies
$$
(W_\i)_t - d \D W_\i = F(v) W_\i \quad\text{in}\quad (\Omega\setminus\overline{B})\times(t_{0},\infty),
$$
while $w: = u_\i ^* - u_\i$ satisfies
\eq{
w_t - d \D w = \left( \frac{f(x,u_\i ^*) - f(x,u_\i)}{u_\i ^* - u_\i} \right) w \leq F(v) w \quad\text{in}\quad (\Omega\setminus\overline{B})\times(t_{0},\infty),
}
where the inequality follows from \eqref{claim-2022-09-17} and the concavity of $f$. Obviously, both $W_{\infty}$ and $w$ satisfy the homogeneous Neumann boundary condition on $\partial\Omega$ and homogeneous Dirichlet boundary condition on $\partial B$. Since $w(\cdot,t_{0})\leq u_{\infty}^{*}\leq M_{2}\hat{\psi}_{\infty}= W_{\infty}(\cdot,t_{0})$, we apply the comparison principle to find $w (\cdot,t) \leq W_\i(\cdot,t)$ for $t\geq t_{0}$. This can be readily extended to hold for all $t\geq0$ by making $M_{2}$ larger if necessary. The conclusion in this step then follows from \eqref{ineq1.9999}.
\end{proof}

We are ready to prove Theorem \ref{finalscalarconvergence}.

\begin{proof}[Proof of Theorem \ref{finalscalarconvergence}]
Clearly, for any $T>0$,
\begin{equation}\label{4.10.1}
\begin{split}
    A_c &:= \sup_{t \in (0,\i)}\norm{u_c(\cdot,t) - u_{\infty}(\cdot,t)}_{C(\overline{\Omega})} \\
&\leq \sup_{t \in (0,T]} \norm{u_{c}(\cdot,t) - u_{\infty}(\cdot,t)}_{C(\overline{\Omega})} + \sup_{t \in (T,\i)} \norm{ u_{c}(\cdot,t) - u_{\infty}(\cdot,t) }_{C(\overline{\O})} \\
&=: A_c ^1 (T) + A_c ^2 (T).
\end{split}
\end{equation}
By Lemma \ref{convcor1}, 
\begin{equation}\label{over-finite-interval-2022-09-19}
\lim_{c \to \i} A_c ^1 (T) = 0,\quad\forall T>0.    
\end{equation}
To treat $A_c ^2 (T)$, we consider two cases. Denote again by $u_0$ the common initial data. 

\paragraph{\bf Case 1: $\mu_{1,\i} > 0$} It follows from Corollary \ref{cor-uniform-converge-to-0} and Theorem \ref{globaldestruction} (ii) that $\lim_{T\to\infty}\lim_{c\to\infty}A_{c}^{2}(T)=0$, which together with \eqref{4.10.1} and \eqref{over-finite-interval-2022-09-19} yields $\lim_{c\to\infty}A_{c}=0$. 

\medskip

\paragraph{\bf Case 2: $\mu_{1,\i} < 0$} Obviously, for any $T>0$,
\eq{
A_c^2 (T) &\leq \sup_{t \in (T,\i)}\norm{ u_c(\cdot,t)-u_c^*}_{C(\overline{\O})} + \norm{ u_c^*-u_\i ^* }_{C(\overline{\O})}+ \sup_{t \in (T,\i)}\norm{ u_\i ^*-u_{\infty}(\cdot,t)}_{C(\overline{\O})} \\
&=: \tilde A_c^1 (T) + \tilde A_c^2 + \tilde A ^{3} (T).
}
By Lemma \ref{uniconvscalar}, there exist $r>0$ and $M=M(u_{0})>0$ such that $\tilde A_c^1 (T) \leq M e^{- r T} + \tilde A_c ^2$ for all $T>0$ and $c\gg1$. 
Since $\lim_{c \to \i} \tilde A_c ^2 = 0$ by Theorem \ref{convlem4} and $\lim_{T\to\infty}\tilde A^3 (T)=0$ by Theorem \ref{globaldestruction} (1), we find $\lim_{T\to\infty}\lim_{c\to\infty}A_{c}^{2}(T)=0$, which together with \eqref{4.10.1} and \eqref{over-finite-interval-2022-09-19} yields $\lim_{c\to\infty}A_{c}=0$.
\end{proof}

\begin{proof}[Proof of Corollary \ref{cor:exthresh}]
    It is an immediate consequence of Theorems \ref{convergence-thm-2}(1) and \ref{globaldegradation} and Corollary \ref{cor-uniform-converge-to-0}. When $\mu_{1,\infty}>0$, the extinction threshold $c_{0}$ is given by $c^{*}$ from Theorem \ref{globaldegradation}. The lower bound on $c_0$ is obtained by noting that when $c = \frac{1}{\as{B}} \int_{\Omega \setminus B} f_u(x,0) {\rm d}x$, $\mu_{1,c} < 0$ by Proposition \ref{deglineig-app}[(iv)]. 
\end{proof}

%%%%%%%%%%%%%%%%%%%%%%%%%%%

\section{Discussion}\label{sec-discussion}

Anthropogenic habitat loss significantly impacts local species, biodiversity, and societies, often through indirect, long-term consequences \cite{Diaz2019, Pimm2014, Pimm2000, jacobson2019global}. Understanding the mechanisms behind species extirpation is crucial.

To this end, we have explored a mathematical framework using reaction-diffusion equations to study the impacts of habitat degradation and destruction. Building on the degradation model from \cite{Salmaniw2022}, we generalized these results to establish a rigorous connection between degradation and destruction. This framework allows for exploring complex relationships between movement strategies, impacts within degraded regions, population growth, and the geometry of degraded/destroyed areas. Empirical data alone cannot fully explain these complex relationships, making mechanistic modeling a valuable complementary strategy. Our uniform convergence results between degradation and destruction problems appear to be the first in the context of spatially explicit models of habitat loss. Some of the convergence results between the eigenvalue problems may be of independent mathematical interest, particularly in handling an unbounded (negative) right hand side.

In addition to the necessary and sufficient conditions for the existence of an extinction threshold, an exponential convergence rate to extirpation beyond the extinction threshold is particularly relevant for practical scenarios of habitat loss. The concept of \textit{extinction debt} \cite{Tilman1994, hanski2002extinction} suggests that consequences of habitat loss may not fully manifest immediately, especially near extinction thresholds. This debt is characterized by a time lag, where the population appears stable over short times but is actually declining. This phenomenon can be understood from our theoretical approach. For $c^*>c_0$ the rate of convergence $r$ found in Theorem \ref{convergence-thm-2} can be identified as precisely $\mu_{1,c^*}$. For any $\varepsilon>0$ we may then compute directly $\mu_{1,c_0 + \varepsilon} = \mu_{1,c_0} - \varepsilon = 0 - \varepsilon = - \varepsilon$. Hence, the rate of population decline can be arbitrarily slow near the extinction threshold $c_0$. While we have shown for a single-species model only, we conjecture similar behavior for multi-species competition models. This leaves open an interesting exploration of habitat loss in the context of biodiversity management.

While habitat fragmentation is beyond the scope of the present work, we acknowledge the importance of studying such mechanistic models with a focus on fragmentation. For empirical studies, we have already applied this model to two experimental designs using the nematode \textit{C. Elegans}. Our spatially explicit approach allows us to examine different configurations of the degraded/destroyed region $B$. While some argue that the total amount of conserved habitat is most critical \cite{Fahrig2013b}, this may not always be the case; in fact, preliminary findings of the forthcoming \cite{Zhang2024b} suggest that locomotion and arrangement plays a key role in the success of a population. As proposed in the introduction, our model invites a detailed analysis of habitat fragmentation effects while keeping the total available habitat fixed, merely rearranging the size and location of patches. While specific to the particular experimental design, these results indicate the promising nature of our modelling approach, at least as applied to species adopting approximately Brownian motion of varying rates as a movement strategy. While this assumption is a limitation of the specific conclusions drawn here, the modelling framework is readily generalized to consider other, possibly more complex, forms of locomotion, or even temporally dynamic landscapes \cite{CantrellCosner2003, Hess1996}. For practical applications, the uniform convergence result between related eigenvalue problems in Theorems \eqref{convergence-thm-2} and \eqref{convergence-thm} allows one to use one model as an approximation of the other; for simulations, problem \eqref{MainEig2-2} more sensible, while problem \eqref{MainEig1-2} may be more appropriate for development of analytical insights. Understanding the convergence rate from $\mu_{1,c}$ to $\mu_{1,\infty}$ as $c$ increases would be of interest.

The mathematical formulation presented here provides a unique opportunity to investigate the consequences of habitat loss on species and biodiversity, including the extinction debt \cite{Tilman1994, Loehle1996, hanski2002extinction}, the habitat-amount hypothesis \cite{Fahrig2013b}, the species-area relationship \cite{Fahrig1996, Fahrig2013b}, and the impact of habitat configuration on survival outcomes \cite{Fahrig2003, Betts1}. We hope to expand on the application of this approach to such questions in future work.

%%%%%%%%%%%%%%%%%%%%%%%%%%%%%%

%%%%%%%%%% Merge with supplemental materials %%%%%%%%%%
\pagebreak
%\widetext
\begin{center}
\textbf{\large Supplemental Materials: From Habitat Decline to Collapse}
\end{center}
%%%%%%%%%% Merge with supplemental materials %%%%%%%%%%
%%%%%%%%%% Prefix a "S" to all equations, figures, tables and reset the counter %%%%%%%%%%
\setcounter{equation}{0}
\setcounter{figure}{0}
\setcounter{table}{0}
\setcounter{section}{0}
\makeatletter
\renewcommand{\thesection}{S\arabic{section}}
\renewcommand{\theequation}{S\arabic{section}.\arabic{equation}}
\renewcommand{\thefigure}{S\arabic{figure}}
%\renewcommand{\bibnumfmt}[1]{[S#1]}
%\renewcommand{\citenumfont}[1]{S#1}
%%%%%%%%%% Prefix a "S" to all equations, figures, tables and reset the counter %%%%%%%%%%

\section{Proof of Lemma \ref{convcor1}}\label{SM:proof_of_thm}

\begin{proof}

Fix $T>0$ and denote by $u_{0}$ the common initial data. The proof is done in four steps.

\paragraph{\bf Step 1} 

We show the existence of some constant $M = M(T)>0$ such that
\eql{
\iint_{Q_T} \left( u_c ^2 + \magg{\grad u_c} + \magg{\frac{\p u_c}{\p t}} \right) \leq M,\quad\forall c\gg1.
}{mainconv1.1}

Due to the lack of smoothness of the solution $u_c$, we first mollify the indicator functions on the right hand side of \eqref{scalareqn-2} so that the approximate solution belongs to $H^1 (Q_T)$. To this end, we set $\e_0: = \frac{1}{2}\text{dist} (\p \O, \p B)$ and define for each $\varepsilon\in(0,\varepsilon_{0})$ the sets:
\eq{
B^\e = \left\{ x \in \overline{\O} : \text{dist} (x, B) < \e \right\}, \quad B_\e = \left\{ x \in B : \text{dist} (x,\p B) > \e\right\}.
}
Note that $B_\e \Subset B \Subset B^\e \Subset \O$. We regularize $\Ind_{\O \setminus B} (x)$ such that
\eq{
&\Ind_{\O \setminus B} ^\e = 1 \text{ in } \O \setminus B^\e, \quad\quad \Ind_{\O \setminus B} ^\e = 0 \text{ in } B, \nonumber \\
&0 \leq \Ind_{\O \setminus B}^\e \leq 1 \text{ in } B^\e \setminus B\quad\text{and} \lim_{\varepsilon\to0}\Ind_{\O \setminus B} ^\e = \Ind_{\O \setminus B}\text{ in } L^2(\O).
}
Similarly, we regularize $\Ind_B(x)$ such that
\eq{
&\Ind_B ^\e = 1 \text{ in } B_\e, \quad \Ind_B ^\e = 0 \text{ in } \O \setminus B, \nonumber \\
& 0 \leq \Ind_B ^\e \leq 1 \text{ in } B \setminus B_\e\quad\text{and}\quad\lim_{\varepsilon\to0}\Ind_B ^\e = \Ind_B \text{ in } L^2(\O).
}
Consider \eqref{scalareqn-2} with $\Ind_{\Omega\setminus B}$ and $\Ind_{B}$ replaced by $\Ind_{\Omega\setminus B}^{\varepsilon}$ and $\Ind_{B}^{\varepsilon}$, respectively, that is
\begin{equation}\label{scalareqn-2-approximate}
\begin{cases}
u_t = d \D u + \Ind_{\O\setminus B}^{\varepsilon} f(x,u) - c \Ind_{B}^{\varepsilon} u, & \text{in}\quad \O\times (0,\i), \\
\frac{\p u}{\p \nu } = 0, &\text{on}\quad \p \O \times (0,\i).
\end{cases}
\end{equation}
Denote by $u_c^\e$ the unique solution of \eqref{scalareqn-2-approximate}
satisfying the initial data $u_{c}^{\varepsilon}(\cdot,0)=u_{0}$. Note that the standard $L^2$-theory of parabolic equations ensures that 
\begin{equation}\label{limit-2022-04-11}
    \lim_{\varepsilon\to0}u_c^\e = u_c\quad\text{in}\quad W^{2,1}_2 (Q_T), 
\end{equation}
and the standard regularity theory ensures that $\frac{\p u^\e_c}{\p t} \in H^1 (Q_T)$. 

We establish some uniform-in-$\varepsilon$ estimates of $u_c^\e$. First, we differentiate\\ $\norm{u_c ^\e (\cdot,t)}_{L^2 (\O)} ^2$ with respect to time and integrate by parts to obtain:
\eq{
\frac{d}{dt} \norm{u_c ^\e(\cdot,t)}_{L^2 (\O)}^2 &= 2 \int_\O u_c ^\e \left(d \D u_c ^\e + \Ind_{G \setminus B} ^\e f(\cdot,u_c ^\e) - c \Ind_B ^\e \right) \\
&\leq 2 \int_\O \as{f(\cdot,u_c ^\e) u_c ^\e} \leq 2 \norm{f_u (\cdot,0)}_{L^\i (\O \setminus\overline{B})} \int_\O (u_c ^\e)^2 .
}
Gr{\"o}nwall's inequality implies that $\norm{u_c ^\e(\cdot,t)}_{L^2 (\O)}^2\leq e^{M_0 T} + \norm{u_0}_{L^2 (\O)}^2$, where $M_0:=2 \norm{f_u (\cdot,0)}_{L^\i (\O \setminus B)}$. Integrating with respect to time yields
\eql{
\norm{u_c ^\e}_{L^2 (Q_T)}^2 &\leq T (e^{M_0 T} + \norm{u_0}_{L^2 (\O)}^2).
}{eqEPSbound}

We now estimate higher order terms. Clearly, $u_{c}^{\varepsilon}$ satisfies
\eql{
\iint_{Q_T} \left( \frac{\p u_c^\e}{\p t} v + d \grad u_c^\e\cdot \grad v \right) = \iint_{Q_T} \left( \Ind_{\O \setminus B}^{\varepsilon} f(\cdot,u_c^\e) - c \Ind_B^{\varepsilon} u_c^\e \right) v ,\quad\forall v \in H^1 (Q_T).
}{weaksoln2}
Setting $v = u_c^\e$ in \eqref{weaksoln2} we have
\eq{
\int_0 ^T \frac{d}{dt} \norm{u_c^\e (\cdot,t)}_{L^2 (\O)}^2 + d \norm{\grad u_c^\e}_{L^2 (Q_T)}^2 &= \iint_{Q_{B,T}} f(\cdot,u_c^\e) u_c^\e - c \iint_{B \times (0,T)} (u_c^\e) ^2 \\
&\leq \norm{f_u (\cdot,0)}_{L^\i (\O\setminus\overline{B})} \norm{u_c^\e}_{L^2 (Q_T)}^2.
}
Note that $\sup_{\varepsilon\in(0,\varepsilon_{0})}\sup_{c\gg1}\norm{u_c^\e}_{L^2 (Q_T)}<\infty$ by \eqref{eqEPSbound}. Setting
$$
M_{1}=M_{1}(T):= \norm{f_u (\cdot,0)}_{L^\i (\O\setminus\overline{B})}\sup_{\varepsilon\in(0,\varepsilon_{0})}\sup_{c\gg1}\norm{u_c^\e}^{2}_{L^2 (Q_T)} + \norm{u_0}_{L^2(\O)}^2, 
$$
we find the uniform bounds
\begin{equation}\label{uniform-bound-1}
    \sup_{0 \leq t \leq T} \norm{u_c^\e(\cdot,t)}_{L^2 (\O)}^2 \leq M_{1}, \quad
\norm{\grad u_c^\e}_{L^2 (Q_T)}^2 \leq \frac{M_{1}}{d},\quad\forall c\gg1\,\,\text{and}\,\,\varepsilon\in(0,\varepsilon_{0}),
\end{equation}

We now obtain estimates on $\frac{\p u_c^\e}{\p t}$. Setting $v = \frac{\p u_c^\e}{\p t}$ in \eqref{weaksoln2}, valid due to the mollification procedure, we deduce
\eq{
\iint_{Q_T} \magg{\frac{\p u_c^\e}{\p t}} &= - \frac{d}{2} \iint_{Q_T} \frac{\p}{\p t} \magg{\grad u_c^\e}  + \iint_{Q_{T}} \Ind_{\O \setminus B}^\e f(\cdot,u_c^\e) \frac{\p u_c^\e}{\p t} - \frac{c}{2} \iint_{Q_T} \Ind_{B}^\e \frac{\p}{\p t}  (u_c^\e) ^2  \\
&\leq -\frac{d}{2} \left( \norm{\grad u_c^\e (\cdot, T)}^2_{L^2 (\O)} - \norm{\grad u_0}^2_{L^2 (\O)} \right) \nonumber \\
&\quad +\frac{1}{2} \iint_{Q_T} \magg{\frac{\p u_c^\e}{\p t}} + \frac{1}{2} \norm{f(\cdot,u_c^\e)}_{L^2 (Q_{B,T})}^2  \\
&\quad - \frac{c}{2} \int_{B} \Ind_B ^\e \left( (u_c^\e) ^2(\cdot,T) - u_0 ^2 \right)  \\
&\leq \frac{d}{2} \norm{\grad u_0}^2_{L^2 (\O)}  + \frac{1}{2} \iint_{Q_T} \magg{\frac{\p u_c^\e}{\p t}} + \frac{1}{2} \norm{f_u (\cdot,0)}^{2}_{L^\i (\O\setminus\overline{B})}\norm{u_c^\e}^{2}_{L^2 (Q_T)},
}
where we have applied Young's inequality, thrown away the negative terms, used the fact that $u_0 \equiv 0$ in $B$. Setting
$$
M_{2}=M_{2}(T):=d\norm{\grad u_0}^2_{L^2 (\O)}  +\norm{f_u (\cdot,0)}^{2}_{L^\i (\O\setminus\overline{B})}\sup_{\varepsilon\in(0,\varepsilon_{0})}\sup_{c\gg1}\norm{u_c^\e}^{2}_{L^2 (Q_T)},
$$
we find
\begin{equation}\label{uniform-bound-2}
\iint_{Q_T} \magg{\frac{\p u_c^\e}{\p t}} \leq M_{2},\quad\forall c\gg1\,\,\text{and}\,\,\varepsilon\in(0,\varepsilon_{0}).
\end{equation}
Passing to the limit $\e \to 0$ in \eqref{uniform-bound-1} and \eqref{uniform-bound-2}, we conclude \eqref{mainconv1.1} from \eqref{limit-2022-04-11}.

%%%%%%%%%%%%%

\paragraph{\bf Step 2}

By \eqref{mainconv1.1}, there is a subsequence, still denoted by $u_c$, and a function $U \in H^1 (Q_{T})$ such that 
\begin{equation}\label{convergence-2022-04-05}
    \begin{split}
        &\lim_{c\to\infty}u_c = U\quad\text{in}\quad L^2 (Q_T),\\
        &\lim_{c\to\infty}\frac{\p u_c}{\p t} = \frac{\p U}{\p t}\quad\text{weakly in}\quad L^2 (Q_T),\\
        &\lim_{c\to\infty}\grad u_c = \grad U\quad\text{weakly in}\quad L^2 (Q_T).
    \end{split}
\end{equation}
Note that in light of Lemma \ref{monotonicallydecreasing}, $U$ must be the pointwise and monotone limit of $u_{c}$ as $c\to\infty$. We show $U=0$ a.e. in $B\times(0,T)$ so that $U\in H_B ^1 (Q_T)$.

Recall that $\phi_{1,c}$ is a positive eigenfunction of \eqref{MainEig2-2} associated with the principal eigenvalue $\mu_{1,c}$. The normalization $\|\phi_{1,c}\|_{L^{2}(\Omega)}=1$ is fixed. Set $w_c: = M e^{-\mu_{1,c} t} \phi_{1,c}$ for some $M>0$ to be determined. Direct computations yield
\eq{
(w_c)_t - d \D w_c = (\Ind_{\O \setminus B} f_u (\cdot,0) - c \Ind_{B}) w_c \geq \Ind_{\O \setminus B}f(\cdot,w_c) - c \Ind_{B} w_c ,
}
where we used in the inequality the fact that $f(\cdot,u) \leq f^\prime (\cdot,0) u$ for any $u\geq0$ due to Assumption \ref{assumptionf} (3). Obviously, $\frac{\partial w_{c}}{\partial\nu }=0$ on $\partial\Omega$.

Theorem \ref{convergence-thm-2} (2) says $\lim_{c \to \i}\phi_{1,c} = \phi_{1,\infty}$ in $H^1 (\O)$, where $\phi_{1,\infty}$ is the positive eigenfunction of \eqref{MainEig1-2} associated with the principal eigenvalue $\mu_{1,\infty}$ and satisfies the normalization $\|\phi_{1,\infty}\|_{L^{2}(\O\setminus \overline{B})} = 1$. This together with $\frac{\partial\phi_{1,\infty}}{\partial \nu }<0$ on $\partial B$ and the conditions on $u_{0}$ ensures the existence of $M\gg1$ such that $u_{0} \leq M \phi_{1,c}$ for all $c\gg1$ by Lemma \ref{eig-lowerbound}. For such a $M$, we apply the comparison principle to conclude that $u_c \leq w_c$ in $\overline{\Omega}\times[0,\infty)$ for all $c\gg1$. This together with Theorem \ref{convergence-thm-2} and the fact $\phi_{1,\infty}\in H_{B}^{1}(\Omega)$ yields
\eq{
\limsup_{c \to \i} \int_0 ^T  \int_B u_c^2 \leq \limsup_{c \to \i}M^2 \int_0 ^T e^{-2 \mu_{1,c} t}dt \int_B \phi_{1,c}^2 = 0.
}
We then conclude from the monotone convergence theorem or the convergence in \eqref{convergence-2022-04-05} that $U=0$ a.e. in $B \times (0,T)$, and hence, $U\in H_B ^1 (Q_T)$. 

%%%%%%%%%%%%%%%%%%%

\paragraph{\bf Step 3}

We show $U=u_{\infty}$ on $\overline{\Omega}\times[0,T]$. Note that $u_{c}$ satisfies
\eql{
\iint_{Q_T} \left( \frac{\p u_c^\e}{\p t} v + d \grad u_c^\e\cdot \grad v \right) = \iint_{Q_T} \left( \Ind_{\O \setminus B}^{\varepsilon} f(\cdot,u_c^\e) - c \Ind_B^{\varepsilon} u_c^\e \right) v ,\quad\forall v \in H^1 (Q_T).
}{weaksoln2-u-c}
For $v \in C^1 ([0,T]; H_B ^1 (\O))$ with $v(T) = 0$, we see from \eqref{weaksoln2-u-c} that
\eql{
- \iint_{Q_{B,T}} \frac{\p v}{\p t} u_c + \iint_{Q_{B,T}} \grad u_c\cdot \grad v = \iint_{Q_{B,T}} f(\cdot,u_c) v + \int_{\O \setminus B} u_0 v(\cdot,0) .
}{init1.1} 

Note from \eqref{weaksoln2-u-c} and \eqref{convergence-2022-04-05} that $U$ satisfies
\eql{
\iint_{Q_{B,T}} \left( \frac{\p U}{\p t} v + d \grad U\cdot \grad v \right) = \iint_{Q_{B,T}} f(\cdot,U) v,\quad\forall v \in H_B^1 (Q_T).
}{weaksoln1}
In particular, for $v \in C^1 ([0,T]; H_B ^1 (\O))$ with $v(T) = 0$,
\eql{
- \iint_{Q_{B,T}} \frac{\p v}{\p t} U + \iint_{Q_{B,T}} \grad U\cdot \grad v = \iint_{Q_{B,T}} f(\cdot,u_\i) v + \int_{\O \setminus B} U (\cdot,0) v(\cdot,0) .
}{init1.2}
Comparing \eqref{init1.1} and \eqref{init1.2} and taking $c \to \i$, we find that indeed $U (\cdot,0) = u_0$ by the arbitrariness of $v(\cdot,0)$.

Consequently, we have shown that $U$ satisfies \eqref{weaksoln1} and $U (\cdot,0) = u_0$. This is actually a weak formulation of \eqref{scalareqn-1}. Moreover, as the pointwise and monotone limit of $u_{c}$ as $c\to\infty$, $U$ must be bounded. We show that the weak formulation admits at most one bounded solution, and then, $U = u_\i$ on $\overline{\Omega}\times[0,T]$.

To this end, we make note of the following fact (see e.g. \cite[Lemma 3.1.2]{Wu2006}): given a function $w \in H_B ^1 (Q_T)$ such that $w(\cdot,0) = 0$, there holds
\eql{
\int_{\O \setminus \overline{B}} w^2 (\cdot, t) = 2 \int_0 ^t \int_{\O \setminus \overline{B}} w \frac{\p u}{\p t} \quad \text{a.e.}\quad t \in (0,T).
}{weakcalc}

Suppose now that there are two bounded solutions $u_1,u_2\in H_B^1 (Q_T)$ satisfying the weak formulation \eqref{weaksoln1} and the same initial data belonging to $H_B^1 (\O) \cap C^{+}(\overline{\O})$, which is assumed to hold in the trace sense. Set $w := u_1 - u_2$ and note that $w(\cdot,0) = 0$ in $\overline{\O} \setminus B$. Then, $w$ satisfies
\eq{
\iint_{Q_{B,T}} \left( \frac{\p w}{\p t} v + d \grad w \cdot \grad v \right)  = \iint_{Q_{B,T}} ( f(\cdot,u_1) - f(\cdot, u_2) ) v,\quad\forall v \in H_B^1 (Q_T).
}
Take $v = w^+ \in H_B ^1 (Q_T)$ and apply \eqref{weakcalc} with the Lipschitz continuity of $f(x,\cdot)$ to obtain
\eq{
\frac{1}{2} \int_{\O \setminus\overline{B}} ( w^+)^2 (\cdot,t) &\leq M \iint_{Q_{B,T}} (w^+) ^2.
}
Gr{\"o}nwall's inequality implies that $\int_{\O \setminus B} (w^+) ^2 = 0 $ for a.e. $t \in (0,T)$. Repeating the procedure for $v = w^-$, we conclude that $w = 0$ a.e. and the uniqueness follows. 

\paragraph{\bf Step 4} As $U$ is the monotone limit of $u_c $ as $c\to\infty$ and $U = u_\i$ is continuous in $\overline{\O}\times[0,T]$ when extended by zero in $B$, we conclude from Dini's theorem that $u_c \to u$ uniformly in $\overline{\O} \times [0,T]$ as $c\to\infty$.
\end{proof}

\section{Eigenvalue problems with sign-indefinite weight}

\ Here we collect some classical results about eigenvalue problems. We are mainly interested in their \emph{principal eigenvalues}, namely, eigenvalues admitting positive eigenfunctions.

We consider the following eigenvalue problem with sign indefinite weight (see e.g. \cite{Brown1980,Ni2001,CantrellCosner2003}):
\begin{equation}\label{MainEig2-app}
\begin{cases}
\D \psi + \l h \psi = 0, & \text{in} \quad \O ,  \\
\frac{\p \psi}{\p \nu } = 0 , & \text{on}\quad \p \O,
\end{cases}
\end{equation}
We point out that \eqref{MainEig2-app} could admit multiple principal eigenvalues, and $0$ is always a principal eigenvalue of \eqref{MainEig2-app}. 

The following result is now standard. Further discussion can be found in \cite{Ni2001,CantrellCosner2003}, for example, with the main result originally obtained in \cite[Theorem 3.13]{Brown1980}. 

\begin{prop}\label{MainEig2-thm-app}
Suppose $h\in L^{\infty}(\O)$ is sign-changing. The following hold.
\begin{enumerate}[\rm(i)]
\item If $\int_\O h \geq 0$, then $0$ is the only non-negative principal eigenvalue of \eqref{MainEig2-app}.

\item If $\int_\O h < 0$, then \eqref{MainEig2-app} admits a unique nonzero principal eigenvalue $\lambda_{1}(h)$, which is simple and given by
\eq{
\l_1 (h) = \inf_{\psi\in H^1 (\O) } \left\{ \frac{\int_\O \magg{\grad \psi}}{\int_\O h \psi^2} :  \int_\O h \psi^2 > 0 \right\}.
}
Moreover, if $h_{1},h_{2}\in L^{\infty}(\Omega)$ are sign-changing and satisfy $h_1 \leq h_2$, then $\l_1 (h_1) \geq \l_1 (h_2)$ with strict inequality whenever $h_1 \not\equiv h_2$.
\end{enumerate}
\end{prop}

This problem, paired with problem \eqref{MainEig2-2-app}, have been applied to a wide variety of biological phenomena, including scalar parabolic problems \cite{Cantrell1991}, the selection of phenotypes based on differing rates of diffusion \cite{Hutson2002}, competition systems in heterogeneous environments with identical resources \cite{He2015} or with equal total resources \cite{He2016,He2017}, and even temporally varying heterogeneous environments, see \cite{Hutson2001} and more recently \cite{Bai2020}.

We formulate the associated \textit{destruction} eigenvalue problem as follows:
\begin{equation}\label{MainEig1-1-app}
\begin{cases}
\D \psi +  \l m  \psi = 0,& \text{in}\quad \O \setminus \overline{B},  \\
\frac{\p \psi}{\p \nu } = 0,& \text{on}\quad\p \O , \\
\psi = 0,& \text{on}\quad \p B , 
\end{cases}
\end{equation}
where $m \in L^\i (\O \setminus{B})$. Note that $0$ is NOT a principal eigenvalue of \eqref{MainEig1-1-app} due to the zero Dirichlet boundary condition on $\partial B$, which actually causes some essential differences between \eqref{MainEig2-app} and \eqref{MainEig1-1-app}.

\begin{prop}\label{existthm-app}
Suppose $m \in L^\i (\O \setminus \overline{B})$ is positive on a set of positive Lebesgue measure. Then, \eqref{MainEig1-1-app} admits a unique positive principal eigenvalue $\l_1 (m,B)$, which is simple and given by
\eq{
\l_1 (m,B) = \inf_{\psi\in H_B ^1 (\O)} \left\{ \frac{\int_{\O\setminus\overline{B}} \magg{\grad \psi}}{\int_{\O\setminus\overline{B}} m \psi^2} : \int_{\O\setminus\overline{B}} m \psi^2 > 0 \right\}.
}
Moreover, $\l_1 (m,B)$ is monotone in the following sense:
\begin{enumerate}[\rm(i)]
\item for any $B_1, B_2$ such that $B_1 \subset B_2$, $\l_1 (m, B_1) \leq \l_1 (m,B_2)$ with strict inequality whenever $B_2 \setminus \overline{B_1}$ has positive measure;
\item for any $m_1,m_2 \in L^\i (\O \setminus B)$ satisfying $m_1 \leq m_2$, $\l_1 (m_1, B) \geq \l_1 (m_2,B)$ with strict inequality whenever $m_1 \not\equiv m_2$.
\end{enumerate}
\end{prop}

Proposition \ref{existthm-app} follows from the more general abstract framework found in \cite[Chapter 3]{Weinberger1974}. See also \cite{Brown1980}, which covers in detail a Dirichlet case on the outer boundary $\p \O$, which is most similar to our problem here.

%%%%%%%%%%%%%%%%%

\section{Eigenvalue problems associated with linearization}\label{appendix-1}
Consider
\begin{equation}\label{MainEig2-2-app}
\begin{cases}
d \D \phi +  h \phi + \mu \phi = 0 ,& \text{in}\quad \O ,  \\
\frac{\p \phi}{\p \nu } = 0 , &\text{on}\quad \p \O,
\end{cases}
\end{equation}
where $h \in L^\i (\O)$. The following result is well-known (see e.g. \cite{Ni2001, CantrellCosner2003}).

\begin{prop}\label{deglineig-app}
The problem \eqref{MainEig2-2-app} admits a unique principal eigenvalue $\mu_1 (d,h)$, which is simple and given by
$$
\mu_1 (d,h) = \inf_{\phi \in H ^1 (\O)} \left\{ \int_\O \left( d \magg{\grad \phi} - h \phi^2 \right)  \ : \ \int_\O \phi^2  = 1  \right\}.
$$
Moreover, $\mu_1 (d,h)$ enjoys the following properties:
\begin{enumerate}[\rm(i)]
\item $d\mapsto\mu_1(d,h)$ is strictly increasing on $(0,\infty)$; 
\item $\mu_1 (d,h_2) < \mu_1 (d,h_1)$ if $ h_2 \gneqq h_1$;
\item $\int_\O h < 0 \Rightarrow$
$\begin{cases}
\mu_1 (d,h) < 0 , &\text{if}\quad d < \frac{1}{\l_1 (h)}, \\
\mu_1 (d,h) = 0 , &\text{if}\quad d = \frac{1}{\l_1 (h)}, \\
\mu_1 (d,h) > 0 , &\text{if}\quad d > \frac{1}{\l_1 (h)} .
\end{cases}$
\item $\int_\O h \geq 0 \Rightarrow \mu_1 (d,h) < 0$ for all $d>0$.
\end{enumerate}
\end{prop}

Similarly, we formulate the associated destruction eigenvalue problem as
\begin{equation}\label{MainEig1-2-app}
\begin{cases}
d \D \phi +  m \phi + \mu \phi = 0 ,& \text{in}\quad \O \setminus \overline{B},  \\
\frac{\p \phi}{\p \nu } = 0 ,& \text{on}\quad \p \O , \\
\phi = 0 ,&  \text{on}\quad \p B, 
\end{cases}
\end{equation}
where $m \in L^\i (\O \setminus \overline{B})$. 

\begin{prop}\label{deglineig-1-app}
The problem \eqref{MainEig1-2-app} admits a unique principal eigenvalue $\mu_1 (d,m,B)$, which is simple and given by
$$
\mu_1 (d,m,B) = \inf_{\phi \in H_{B} ^1 (\O)} \left\{ \int_{\O\setminus\overline{B}} \left( d \magg{\grad \phi} - m \phi^2 \right)  \ : \ \int_{\O\setminus \overline{B}} \phi^2  = 1  \right\}.
$$
Moreover, $\mu_1 (d,m,B)$ enjoys the following properties:
\begin{enumerate}[\rm(i)]
\item $d\mapsto\mu_1 (d,m,B)$ is strictly increasing on $(0,\infty)$; 
\item $\mu_1 (d,m_2,B) < \mu_1 (d,m_1,B)$ if $ m_2 \gneqq m_1$;
\item $m \leq 0 \Rightarrow \mu_1 (d,m,B) > 0$ for all $d > 0$;
\item $m > 0$ on some nontrivial subset $\Rightarrow$
$\begin{cases}
\mu_1 (d,m,B) < 0, &\text{if}\quad d < \frac{1}{\l_1 (m,B)}, \\
\mu_1 (d,m,B) = 0, &\text{if}\quad d = \frac{1}{\l_1 (m,B)}, \\
\mu_1 (d,m,B) > 0, &\text{if}\quad d > \frac{1}{\l_1 (m,B)}.
\end{cases}$
\item If $m_n \to m$ in $C(\overline{\O} \setminus B)$, then $\mu_1 (d,m_n,B) \to \mu_1 (d,m,B)$ as $n \to \i$.
\end{enumerate}
\end{prop}

%%%%%%%%%%%%%%%%%%%%

\section{Uniform bounds of principal eigenfunctions}\label{appendix-uniform-bounds}

Denote by $\mu_{1,c}$ the principal eigenvalue with eigenfunction $\phi_c$ solving problem \eqref{MainEig2-2-app} with\\ $h = \Ind_{\O \setminus B} m - c \Ind_B$ for some $m \in L^\i (\O \setminus B)$, normalized so that $\norm{\phi_c}_{L^2 (\O)} = 1$. We include the following technical lemmas which give some uniform boundedness estimates from above and below on $\phi_{c}$ with respect to $c\gg1$.

\begin{lem}\label{eig-lowerbound}
Given any subset $K \Subset \overline{\O} \setminus \overline{B}$, there holds
\eq{
0 < \inf_{c \gg 1} \inf_{K} \phi_c \leq \sup_{c \gg 1} \sup_{K} \phi_c < \i.
}
\end{lem}
\begin{proof}
By a slight modification of the proof of Theorem \ref{convergence-thm-2}, 
\begin{equation}\label{limit-appendix}
\lim_{c\to\infty}\phi_c = \phi_\i\quad\text{in}\quad H ^1 (\O),    
\end{equation}
where $\phi_\i \in H_B ^1 (\O)$ is the first eigenfunction solving problem \eqref{MainEig1-2-app} normalized so that $\norm{\phi_\i}_{L^2 (\O \setminus B)} = 1$. Since $m \in L^\i (\O \setminus B)$ and $\phi_\i = 0$ on $\p B$ (in the sense of the trace), $L^p$-theory of elliptic equations guarantees that $\phi_\i \in C( \overline{\O} \setminus B)$.

Without loss of generality, we may assume $K$ has a smooth boundary. Then, from standard $L^2$-theory of elliptic equations (see, e.g., \cite[Chapter 2.2.2]{Wu2006}), $\{\phi_c \}_{c \gg 1}$ is bounded in $H^2 (K)$. Applying the usual bootstrapping arguments via $L^p$-estimates for elliptic equations (see, e.g., \cite[Ch. 9.4]{SOB2}), we have that in fact $\{ \phi_c \}_{c \gg 1}$ is bounded in $W^{2,p} (K)$ for any $p \geq 1$, since $m \in L^\i (\O \setminus B)$ does not depend on $c$. By the Sobolev embedding, $\{ \phi_c \}_{c \gg 1}$ is bounded in $C^{1,\a} (\overline{K})$ for some $\a \in (0,1)$, and so, $\lim_{c\to\infty}\phi_c = \phi_\i$ in $C (\overline{K})$ thanks to the Arzel\`a-Ascoli theorem and \eqref{limit-appendix}. Since $0 <\inf_{K} \phi_c \leq \sup_{K} \phi_c < \i$, the conclusion of the lemma follows.
\end{proof}

\begin{lem}\label{eig-upperbound}
There holds $\sup_{c \gg 1} \sup_{\O} \phi_c < \i$.
\end{lem}
\begin{proof}
From Lemma \ref{eig-lowerbound}, we see that $\{ \phi_c \}_{c \gg 1}$ is uniformly bounded from above for any $K \Subset \overline{\O} \setminus \overline{B}$. The delicacy in this case comes in deriving a uniform upper bound on $\phi_c$ in a neighbourhood of $B$. Unlike the previous case, we cannot apply the same $L^p$ style arguments since $h = \Ind_{\O \setminus B} m - c \Ind_B$ becomes unbounded in $L^p (\O)$ as $c\to\infty$ for any $p \geq 1$. For this reason we appeal to an application of the Moser iteration technique. To this end, we seek to obtain a bound of the form
\eql{
\norm{\phi_c}_{L^{2N^{k}/(N-2)^{k}} (B_{R_{k+1}} (x_0))} \leq C_k 
\norm{\phi_c}_{L^{2N^{k-1}/(N-2)^{k-1}}(B_{R_{k}} (x_0))},
}{t1.0}
for some constants $C_k$ such that their product $\prod_{n=1}^\i C_n$ is bounded independent of $c\gg1$, and $B_{R_{k+1}} (x_0) \Subset B_{R_k} (x_0)$ concentric balls of particular radii $R_k$ defined below. In the above estimate, $N\geq3$ is the spatial dimension. The cases $N=1,2$ are simpler and the details are omitted.

\paragraph{\bf Step 1} By Theorem \ref{convergence-thm-2}, $\{\mu_{1,c}\}_{c\gg1}$ is bounded and $\lim_{c\to\infty}\phi_c = \phi_\i$ in $H^{1}(\Omega)$ for some $\phi_\i \in H_B ^1 (\O)$, considered as an element in $H^{1}(\Omega)$ by zero extension. 

Since $B \Subset \O$, there is $R>0$ such that the $R$-neighbourhood of $B$ is compactly contained in $\Omega$. Fix an arbitrary point $x_0 \in \overline{B}$. Then, $B_R (x_0) \Subset \O$. We drop the dependence on $x_0$ moving forward for notational brevity. Choose a cutoff function $\eta \in C_0 ^\i (B_R)$ so that $0 \leq \eta \leq 1$ in $B_R$, $\eta = 1$ in $B_{R(1-1/N)}$, and $\as{\grad \eta} \leq 4 N^2 / R(N-2)$. Multiplying the equation for $\phi_c$ by $\eta^2 \phi_c$ and integrating by parts yields
\eql{
d \int_{B_R} \magg{\grad \phi_c} \eta^2 &\leq 2 d \int_{B_R} \eta \as{\grad \phi_c} \as{\grad \eta} \phi_c + \int_{B_R} \eta^2 ( \Ind_{\O \setminus B} m - c \Ind_B + \mu_{1,c} ) \phi_c ^2.
}{t1.1}
Applying Young's inequality to the first term on the right hand side of \eqref{t1.1} side yields
\eq{
2d \int_{B_R} \eta \as{\grad \phi_c} \as{\grad \eta} \phi_c &\leq \frac{d}{2} \int_{B_R} \eta^2 \magg{\grad \phi_c} + 2d \int_{B_R} \phi_c^2 \magg{\grad \eta} \\
&\leq \frac{d}{2} \int_{B_R} \eta^2 \magg{\grad \phi_c} + 2d \left( \frac{4N^2}{R(N-2)} \right)^2 \int_{B_R} \phi_c^2 .
}
Combining this with \eqref{t1.1}, using the boundedness of $m$, $\{\mu_{1,c}\}_{c\gg1}$, and dropping the negative term, we are left with
\begin{equation}\label{t1.t1}
    \frac{d}{2} \int_{B_R} \magg{\grad \phi_c} \eta^2 \leq \left( 2d \left( \frac{4N^2}{R(N-2)} \right)^2 + \norm{m}_{L^\i (\O \setminus B)} + \as{\mu_{1,c}} \right) \int_{B_R } \phi_c^2 \leq C_0 \int_{B_R} \phi_c ^2 .
\end{equation}
Since $\eta \phi_c \in H_0 ^1 (B_R)$, the Sobolev inequality and Poincar\'e's inequality yields
\eq{
\norm{\eta \phi_c}_{L^{2N/(N-2)} (B_R)} &\leq C \norm{\eta \phi_c}_{H^1 (B_R)} \leq C \norm{\grad ( \eta \phi_c )}_{L^2 (B_R)} ,
}
where $C$ may change between inequalities but does not depend on $c\gg1$. Using the fact that $\grad ( \eta \phi_c ) = \grad \eta  \phi_c + \grad \phi_c \eta$ paired with the estimate \eqref{t1.t1}, we see that
\eq{
\frac{d}{4}\norm{\grad (\eta \phi_c)}_{L^2 (B_R)}^2 &\leq \frac{d}{2} \int_{B_R} ( \magg{\grad \eta} \phi_c^2 + \magg{\grad \phi_c} \eta^2 )   \\
&\leq \frac{d}{2} \left( \frac{4N^2}{R(N-2)} \right)^2 \int_{B_R} \phi_c ^2 + C_0 \int_{B_R} \phi_c ^2 \leq C_0 \norm{\phi_c}_{L^2 (B_R)}^2,
}
where again $C_0$ may change from line to line but remains independent of $c\gg1$. Finally, using the fact that $\eta = 1$ in $B_{R(1-1/N)}$ we obtain the estimate
\eql{
\norm{\phi_c}_{L^{2N/(N-2)} (B_{R(1-1/N)})} &\leq C_1 \norm{\phi_c}_{L^2 (B_R)} ,
}{step1}
where $C_1$ depends on all quantities thus far but can be chosen independent of $c\gg1$. 

\paragraph{\bf Step 2} We now show the induction step. Set $\a_k = (N/(N-2))^{k-1}$ for integer $k \geq 1$ and consider the sequence of radii $R_{k} = \frac{R}{2}( 1 + \a_k ^{-1})$ so that $R_1 = R$ and $R_\i := \lim_{k \to \i} R_k = R/2$. Note that we have established \eqref{t1.0} for $k=1$ (namely, \eqref{step1}), where $C_1$ is as defined above. Then, we consider a sequence of cutoff functions $\eta_k \in C_0 ^\i (B_{R_k})$ so that $0 \leq \eta_k \leq 1$, $\eta_k = 1$ in $B_{R_{k+1}}$, and $\as{\grad \eta_k} \leq 4/(R_{k} - R_{k+1}) = 4 N \a_k / R$. Multiplying the equation for $\phi_c$ by $\eta^2 \phi_c ^{2\a_k -1}$, integrating by parts and throwing away negative terms yields
\eql{
\frac{d(2 \a_k - 1)}{\a_k ^2} \int_{B_{R_k}} \magg{\grad \phi_c ^{\a_k}} \eta_k^2 &\leq 2 d \int_{B_{R_k} } \eta_k \as{\grad \phi_c} \as{\grad \eta_k} \phi_c ^{2 \a_k - 1} \nn \\
&\quad+ \left( \norm{m}_{L^\i (\O \setminus B)} + \as{\mu_{1,c}} \right) \int_{B_{R_k} } \phi_c ^2 .
}{t1.4}
We again control the first term on the right hand side via Young's inequality and absorb into the left hand side. To this end, we compute
\eq{
2 d \int_{B_{R_k}} \eta_k \as{\grad \phi_c} \phi_c ^{2 \a_k - 1} \as{\grad \eta_k} &= \frac{2 d}{\a_k} \int_{B_{R_k} } \eta_k \as{\grad \phi_c ^{\a_k}} \phi_c ^{\a_k} \as{\grad \eta_k}  \\
\leq \frac{d (2 \a_k - 1)}{2 \a_k ^2}& \int_{B_{R_k} } \magg{\grad \phi_c ^{\a_k}} \eta_k ^2 + \frac{2d}{2 \a_k - 1} \int_{B_{R_k} } \phi_c ^{2 \a_k} \magg{\grad \eta_k} \\
\leq \frac{d (2 \a_k - 1)}{2 \a_k ^2} &\int_{B_{R_k} } \magg{\grad \phi_c ^{\a_k}} \eta_k ^2 + \frac{32 d N^2 \a_k ^2}{R^2 (2 \a_k -1)} \int_{B_{R_k} } \phi_c ^{2 \a_k } .
}
Combining this result with \eqref{t1.4} leaves
\eql{
\frac{d (2\a_k -1)}{2 \a_k ^2} \int_{B_{R_k}} \magg{\grad \phi_c ^{\a_k}} \eta_k ^2 &\leq \left( \norm{m}_{L^\i (\O \setminus B)} + \as{\mu_{1,c}} + \frac{32 d N^2 \a_k ^2}{R^2 (2 \a_k -1)} \right) \norm{\phi_c ^{\a _k}}_{L^2 (B_{R_k})} ^2.
}{t1.7}
Notice again that $\eta_k \phi_c ^{\a_k}$ belongs to $H_0 ^1 (B_{R_k})$. Therefore, applying the Sobolev inequality, Poincar\'e's inequality and the fact that $\grad (\eta_k \phi_c ^{\a_k}) = \grad \eta_k \phi_c ^{\a_k} + \eta_k \grad ( \phi_c ^{\a_k})$ gives us that
\eq{
\frac{d (2 \a_k - 1)}{4 \a_k ^2} \norm{\eta_k \phi_c ^{\a_{k}}}^2_{L^{2N/(N-2)} (B_{R_{k}})} &\leq \frac{d (2 \a_k - 1)}{2 \a_k ^2} \int_{B_{R_k}} \left( \magg{\grad \eta_k} \phi_c ^{2 \a_k} + \magg{\grad \phi_c ^{\a_{k}}} \eta_k ^2 \right),
}
and so combining this estimate with \eqref{t1.7} and using that $\eta_k \equiv 1$ in $B_{R_{k+1}}$ yields
\eql{
&\tfrac{d (2 \a_k - 1)}{4 \a_k ^2} \norm{\phi_c ^{\a_k}}^2_{L^{2N/(N-2)} (B_{R_{k+1}})} \leq \nonumber \\
& \left( \norm{m}_{L^\i (\O \setminus B)} + \as{\mu_{1,c}} + \tfrac{32 d N^2 \a_k ^2}{R^2 (2 \a_k - 1)} + \tfrac{8 d N^2 (2 \a_k - 1)}{R^2}  \right) \norm{\phi_c^{\a_k}}^2_{L^2 (B_{R_k})}.
}{eq1111111}
An elementary manipulation gives that
\eql{
\norm{\phi_c ^{\a_k}}_{L^{2N/(N-2)} (B_{R_{k+1}} )}= \norm{\phi_c}_{L^{2\a_{k+1}} (B_{R_{k+1}} )} ^{\a_k},\quad\norm{\phi_c ^{\a_k}}_{L^2 (B_{R_k} )}= \norm{\phi_c}_{L^{2\a_{k}} (B_{R_k} )} ^{\a_k} .
}{eq2222222222}
Finally, rearranging \eqref{eq1111111} and using \eqref{eq2222222222} we obtain the final estimate
\eq{
\norm{\phi_c}_{L^{2\a_{k+1}} (B_{R_{k+1}} )} &\leq C_k \norm{\phi_c}_{L^{2 \a_k} (B_{R_k} )},
}
where $C_k$ is a constant depending on all quantities used throughout this procedure but can be chosen independent of $c$, and is dominated by a term of order $(\a_k ^4/ (2 \a_k -1)^2)^{1/2\a_k} \sim ( \a_k )^{1/\a_k}$ for $k$ large.

\paragraph{\bf Step 3} We complete the limiting process. The uniformity in $c$ is clear; on the other hand, upon iteration we find that
\eql{
\norm{\phi_c}_{L^{2N^k / (N-2)^k} (B_{R_{k+1}} )} &\leq \prod_{n=1}^k C_n \norm{\phi_c}_{L^2 (B_{R_k} )} ,
}{t.18}
and so we now ensure that the product of the constants $C_k$ are bounded. First, note that there exists a constant $A$ depending on $\norm{m}_{L^\i (\O)}, \as{\mu_{1,c}}, d, N, R$ but independent of $c,k$ so that
\eq{
C_k \leq \left( A \a_k \right)^{1/{\a_k}}.
}
Then, we use the fact that $\prod_{n=1}^\i C_n < \i \iff \sum_{n=1}^\i \log (C_n) < \i$. Using the bound above and some elementary calculation, we see that
\eq{
\sum_{n=1}^\i \log (C_n) &\leq \sum_{n=1}^\i \frac{(n-1) \log (A^{1/(n-1)} \s )}{\s ^{n-1}} < \i,
}
where $\s = N/(N-2) > 1$ ensures the convergence. Thus, \eqref{t.18} is bounded, and taking $k \to \i$ yields
\eq{
\norm{\phi_c}_{L^\i (B_{R/2} (x_0))} &\leq M \norm{\phi_c}_{L^2 (B_R (x_0))} .
}
Since $x_0 \in \overline{B}$ was arbitrary, we have that $\phi_c$ is uniformly bounded on some set $B^\prime$ such that $B \Subset B^\prime$. Combining this with Lemma \ref{eig-lowerbound}, we conclude that $\sup_{c \gg 1} \sup_{\O} \phi_c$ is bounded.
\end{proof}

\bibliographystyle{siamplain}
\bibliography{ex_article}

\end{document}